\documentclass{amsart}
\usepackage{amssymb}
\usepackage{amscd}
\usepackage{enumitem}
\usepackage{color}
\usepackage{tikz-cd}

\usepackage{hyperref}

\newtheorem{theorem}{Theorem}[section]
\newtheorem{lemma}[theorem]{Lemma}

\newtheorem{definition}[theorem]{Definition}
\newtheorem{example}[theorem]{Example}
\newtheorem{remark}[theorem]{Remark}
\newtheorem{proposition}[theorem]{Proposition}
\newtheorem{corollary}[theorem]{Corollary}
\newtheorem{notation}[theorem]{Notation}

\newcommand\ec { \color{black}}

\begin{document}

\title[Factorization in almost Dedekind domains]
{Factorization in almost Dedekind domains} 
\author[G.W. Chang]{Gyu Whan Chang}

\address{(Chang) Department of Mathematics Education, Incheon National University,
Incheon 22012, Korea} \email{whan@inu.ac.kr}

\author[H.S. Choi]{Hyun Seung Choi}

\address{(Choi) Research Institute of Basic Sciences, Incheon National University,
Incheon 22012, Korea} \email{hchoi21@inu.ac.kr}

\date{\today}

\thanks{2000 Mathematics Subject Classification: 11T06, 13A05, 13A15, 13B25, 13F05, 13G05, 13P05}
\thanks{Key Words and Phrases: Almost Dedekind domain, PID, polynomial ring, DVR, B{\'e}zout domain, irreducible polynomial, cyclotomic polynomial}

\begin{abstract}
Let $F$ be a field, $p$ a prime number, $X$ an indeterminate over $F$,
$D_n =F[X^{\frac{1}{p^n}}, X^{-\frac{1}{p^n}}]$ for each integer $n \geq 0$ and
$D = \bigcup\limits_{n\in\mathbb{N}_0}D_n.$ Then $D$ is a one-dimensional B{\'e}zout domain
but not a Dedekind domain, and $D$ is an almost Dedekind domain if and only if char$(F) \neq p$.
In this paper, we study the element-wise factorization properties of $D$. For example, we determine when an irreducible element of $D_n$ is an irreducible element of $D$, in terms of $n$ and $p$.
In particular, we show that if $F$ is algebraically closed or
a finite field of char$(F)=p$, then $D$ has no irreducible element.
We also show that if $F$ is a finite field of odd characteristic, then an irreducible element
$f(X)$ of $D_0$ is irreducible in $D$ if and only if it is a factor of a cyclotomic polynomial $\Phi_n(X)$ for some integer $n \geq 1$ which satisfies a certain equation in
terms of $|F|$ and $\textnormal{deg}(f(X))$.
Finally, we introduce the notion of infinite product and we then show that
if $F= \mathbb{Q}$ and $p=2$, every nonzero
nonunit of $D$ can be written as a product of countably many prime elements of $D$
and every proper nonzero principal ideal of $D$ can be uniquely written as a countable intersection of principal primary ideals.
\end{abstract}

\maketitle

\section{Introduction}

Let $D$ be an integral domain with quotient field $K$.
An \textit{overring} of $D$ means a subring of $K$ containing $D$.
We say that $D$ is a \textit{discrete valuation ring} (\textit{DVR}) if $D$ is a principal
ideal domain (PID) with exactly one nonzero maximal ideal.
An \textit{almost Dedekind domain} $D$ is an integral domain in which $D_M$ is a DVR for all
maximal ideals $M$ of $D$.
We know that a nonzero finitely generated ideal of $D$ is invertible
if and only if it is locally principal \cite[Corollary 7.5]{gilmer}.
Hence, $D$ is a Dedekind domain if and only if $D$ is a Noetherian almost Dedekind domain,
if and only if $D$ is an almost Dedekind domain and each nonzero element of $D$ is contained in
only finitely many maximal ideals of $D$ \cite[Theorems 37.1 and 37.2]{gilmer}.
It is also easy to see that an overring of an almost Dedekind domain is an almost Dedekind domain \cite[Corollary 36.3]{gilmer}.

In the next subsection, we briefly introduce the results of this paper.
We also give a historical overview on almost Dedekind domains and
a couple of notations for easy understanding of the result of this paper.

\subsection{Results}
Let $\mathbb{N}$ be the set of positive integers and $\mathbb{N}_0 = \mathbb{N} \cup \{0\}$.
Let $F$ be a field, $p$ a prime number, $X$ an indeterminate over $F$, $D_n = F[X^{\frac{1}{p^n}}, X^{-\frac{1}{p^n}}]$
for each $n \in \mathbb{N}_0$ and $D = \bigcup\limits_{n \in \mathbb{N}_0}D_n$.
It is easy to see that $D_n$ is a PID, $D_n \subsetneq D_{n+1}$ for every $n \in \mathbb{N}_0$
and $D$ is a one-dimensional B{\'e}zout domain. In this paper, we study the ring-theoretic and arithmetic properties
of $D$ with a focus on the almost Dedekind domain property and the irreducibility of an element of $D$.

The paper consists of six sections including the introduction. Starting with general results,
we gradually develop theories and examples that apply to specific domains, adding restrictions one by one. In Section 1.2, we give a short historical overview on how to
construct an almost Dedekind domain that is not Dedekind via the so-called Dedekind
tower-construction. We also recall when a semigroup ring is almost Dedekind
by which we can show that $D = \bigcup\limits_{n \in \mathbb{N}_0}D_n$ is
an almost Dedekind domain if and only if char$(F) \neq p$.
In Section 2, we give some basic properties
of an integral domain $D$ constructed from a chain of Dedekind domains
with additional properties (see Notation \ref{notation1}).
Specifically, we show that $D$ is a one-dimensional Pr\"ufer domain;
we give a necessary and sufficient condition for $D_M$ to be a DVR for a maximal ideal $M$ of $D$
and for $D$ to be an almost Dedekind domain or a Dedekind domain.
We calculate ht$(M[[X]])$ for a maximal ideal $M$ of $D$, showing that
dim$(D[[X]]) \geq 2^{\aleph_1}$ if and only if $D$ is not a Dedekind domain.

Let $D = \bigcup\limits_{n \in \mathbb{N}_0}F[X^{\frac{1}{p^n}}, X^{-\frac{1}{p^n}}]$.
In Section 3, we study several interesting properties of $D$.
In particular, we note that $D$ has no irreducible element
if either $F$ is algebraically closed or char$(F) =p$ with $F$ algebraic over $\mathbb{Z}/p\mathbb{Z}$. Let $F$ be an algebraically closed field of characteristic $p$.
We also classify the set of maximal ideals of $D$
and show that $D_P$ is a rank one nondiscrete valuation domain and $D/P \cong F$
for all maximal ideals $P$ of $D$.

Section 4 is devoted to the study of the prime ideal structure and factorization property of $D$ under the setting $p=2$,
i.e., of $D = \bigcup\limits_{n \in \mathbb{N}_0}F[X^{\frac{1}{2^n}}, X^{-\frac{1}{2^n}}]$. For example,
we prove that char$(F) \neq 2$ if and only if  every proper principal ideal of $D$ is a finite product of radical principal ideals of $D$.
We also investigate when $f(X^{\frac{1}{2^n}}) \in D_n$ is a prime element of $D$,
the structure of Max$(D)$ when $D$ is an algebraically closed field with char$(F) \neq 2$
and the monic polynomials in $F[X]$ that are prime elements of $D$
when $|F| < \infty$ and char$(F) \neq 2$.

Finally, in Section 5, we introduce the notion of infinite products of prime elements
and we then use this notion to prove that every nonzero nonunit of
$D = \bigcup\limits_{n \in \mathbb{N}_0}\mathbb{Q}[X^{\frac{1}{2^n}}, X^{-\frac{1}{2^n}}]$
can be written as a product of finite or countably infinite number of prime elements of $D$.
Of course, $D$ is not a UFD. We also show that every proper nonzero principal ideal of $D$
can be uniquely written as a countable intersection of principal primary ideals.
Conversely, given a collection of countably many pairwise comaximal principal primary ideals $\{Q_i\}_{i\in\mathbb{N}}$ of $D$, we completely determine when $\bigcap\limits_{i\in\mathbb{N}}Q_i$ is a principal ideal of $D$.

\subsection{Historical overview}
Almost Dedekind domains form an important subclass of the class of Pr\"ufer domains
which includes Dedekind domains, i.e.,
\begin{center}
Dedekind domains $\Rightarrow$ almost Dedekind domains $\Rightarrow$ Pr\"ufer domains.
\end{center}
It has many fascinating applications and we will briefly mention some of them in this section.
We refer to Loper's article \cite{lo} for a detailed exposition.

Dedekind domains and the concept of ideals were introduced to better understand the factorization in an integral domain that is not necessarily a UFD.
In particular, an integral domain $D$ is a Dedekind domain if and only if every proper ideal of $D$ is a finite product of prime ideals,
if and only if every nonzero ideal of $D$ is invertible. In terms of ideal arithmetic,
the most convenient property of invertible ideals is that they are cancellative.
(A nonzero ideal $I$ of an integral domain $D$ is \textit{cancellative} if $IJ=IH$ implies $J=H$ for ideals $J$ and $H$ of $D$.)
This is where the class of almost Dedekind domains enters the stage.

\begin{theorem}
\label{th0}
\cite[Theorem 36.5]{gilmer}
An integral domain $D$ is an almost Dedekind domain if and only if every nonzero ideal of $D$ is cancellative.
\end{theorem}

On the other hand, let $D$ be an integral domain with quotient field $K$, $K[X]$ the polynomial ring over $K$
 and Int$(D)=\{f(X)\in K[X]\mid f(D)\subseteq D\}$.
Then Int$(D)$ is an overring of $D[X]$ and called the ring of \textit{integer-valued polynomials} of $D$.
Polya \cite{p19} and Ostrowski \cite{o19} first investigated Int$(D)$ when $D$ is the ring of integers of an algebraic number field.
One motivating question was the ring-theoretic properties of Int$(\mathbb{Z})$
and Brizolis \cite{b79} proved that Int$(\mathbb{Z})$ is a Pr\"ufer domain.
Then the focus was on the classification of the integral domain $D$ such that Int$(D)$ is Pr\"ufer.
Chabert proved that if Int$(D)$ is a Pr\"ufer domain, then $D$ is an almost Dedekind domain whose residue fields are finite \cite[Proposition 6.3]{c87},
while Gilmer showed that the converse fails in general \cite[Theorem 13]{g90}.
For a complete characterization of the integral domain $D$ such that Int$(D)$ is Pr\"ufer, see \cite{l97}.
This briefly shows why almost Dedekind domains matter. The interested reader is guided to the textbook of Cahen and Chabert \cite{cc97}.

Now, let us return to some fundamental questions: Is there an almost Dedekind domain that is not Dedekind?
If so, then is there a canonical way to construct such an integral domain? The answer is well-known to be affirmative,
as we will see shortly. There are several ways to construct an almost Dedekind domain, and in this paper,
we will focus on the ones motivated by the following result of Nakano \cite{n53}:
forming an ascending chain of Dedekind domains satisfying certain properties and taking the union of them.
Nakano's construction is the first known example of non-Noetherian almost Dedekind domains.

\begin{theorem} \label{nakano} \cite[p. 426]{n53}
Let $\mathcal{P}=\{p_i\}_{i\in\mathbb{N}}$ be the set of prime numbers
and $\zeta_i$ the primitive $p_i$-th root of unity for each $p_i\in\mathcal{P}$. Then for each $n\in\mathbb{N}$,
let $K_n=\mathbb{Q}(\zeta_1,\dots, \zeta_n)$ and $R_n$ the integral closure of $\mathbb{Z}$ in $K_n$.
Then $R=\bigcup\limits_{n\in\mathbb{N}}R_n$ is an almost Dedekind domain that is not Dedekind.
\end{theorem}

This type of construction was generalized by Butts and Yeagy \cite{by76} in the following sense:
Let $R_1\subseteq R_2\subseteq\cdots$  be a chain of Dedekind domains with quotient fields $K_1\subseteq K_2\subseteq\cdots$,
respectively, such that for each $i\in\mathbb{N}$,
\begin{itemize}
\item $R_{i+1}$ is integral over $R_i$ and
\item $K_{i+1}$ is a finite separable extension of $K_i$ with $K_{i+1}\neq K_i$.
\end{itemize}
Then $R_1\subseteq R_2\subseteq \cdots \subseteq R$, where $R=\bigcup\limits_{i\in\mathbb{N}}R_i$, is called
a \textit{Dedekind tower-constuction}.
Butts and Yeagy provided a sufficient condition for a Dedekind tower-construction to yield an almost Dedekind domain.

\begin{theorem} \label{bytheo}
\cite[Lemma 8]{by76} Suppose that $R_1\subseteq R_2\subseteq \cdots \subseteq R$ is a Dedekind tower-construction.
Suppose also that for every maximal ideal $P$ of $R$, there exists an integer $I(P)\in\mathbb{N}$ such that
$(P\cap R_{i})\not\subseteq (P\cap R_{i+1})^2$ for all integers $i\ge I(P)$. Then $R$ is an almost Dedekind domain.
\end{theorem}

Loper and Lucas \cite[Theorem 2,10]{ll03} further relieved the restriction on Dedekind tower-construction in the following form.

\begin{theorem}
Suppose that $R_1\subseteq R_2\subseteq \cdots $ is a chain of Dedekind domains satisfying each of the following
three conditions for all $j \in \mathbb{N}$:
\begin{enumerate}[font=\normalfont]
\item Each maximal ideal of $R_i$ survives in $R_j$ for all $i \in \mathbb{N}$ with $i<j$.
\item Each maximal ideal of $R_j$ contracts to a maximal ideal of $R_1$.
\item If $M'$ is a maximal ideal of $R_j$ and $M=M'\cap R_1$, then $M(R_j)_{M'}=M'(R_j)_{M'}$.
\end{enumerate}
Let $D=\bigcup\limits_{i\in\mathbb{N}}R_i$.
Then $D$ is an almost Dedekind domain.
\end{theorem}

 Let $S$ be a torsion-free commutative cancellative monoid,
$E$ an integral domain, and $E[X;S]$ the monoid
domain of $S$ over $E$. Then $E[X;S]$ is an integral domain.
In \cite{gp74}, Gilmer and Parker sought conditions on $E$ and $S$ under which
$E[X;S]$ will have a given ring-theoretic property; for example, Pr\"ufer ring,
almost Dedekind ring or a general ZPI-ring in a more general setting of
rings with zero divisors. Among them, we recall the following useful result
for which we let $\mathbb{Q}$ be the additive group of rational numbers,
$\mathbb{N}_0$ the additive monoid of nonnegative integers and $\mathbb{Z}$
the additive group of integers.

\begin{theorem} \label{gp} \cite[Theorem (5)]{gp74}
Let $S$ be a torsion-free commutative cancellative monoid, $E$ an integral domain
and $E[X;S]$ the monoid domain of $S$ over $E$. Then
$E[X;S]$ is an almost Dedekind domain
if and only if $E$ is a field, $S$ is
isomorphic to either $\mathbb{N}_0$ or a subgroup of $\mathbb{Q}$ containing $\mathbb{Z}$
such that if char$(E)=p$ is nonzero, then $\frac{1}{p^k} \not\in S$ for some $k\in\mathbb{N}$.
\end{theorem}

As Loper mentioned in \cite[p.291]{lo}, the almost Dedekind monoid domain in Theorem \ref{gp} can be described as a union of integral extensions of $F[X, X^{-1}]$ for a field $F$ by adding roots of $X$ and it seems to possess several interesting factorization properties.
The following result is a special case of Theorem \ref{gp}, which is the main subject of this paper.

\begin{corollary} \label{thm31}
Let $p$ be a prime number, $F$ a field, $X$ an indeterminate over $F$,
$D_n=F[X^{\frac{1}{p^n}}, X^{-\frac{1}{p^n}}]$ for each $n\in\mathbb{N}_0$,
and $D=\bigcup\limits_{n\in\mathbb{N}_0}D_n$.
\begin{enumerate}[font=\normalfont]
\item If $G$ is the subgroup of $\mathbb{Q}$ generated by
$\{\frac{1}{p^n} \mid n \in \mathbb{N}_0\}$, then $D = F[X;G]$.
\item $D$ is an almost Dedekind domain
if and only if char$(F) \neq p$.
\end{enumerate}
\end{corollary}

\begin{proof}
(1) Clear.

(2) Let $q =$ char$(F)$ and $G$ the subgroup of $\mathbb{Q}$ generated by
$\{\frac{1}{p^n} \mid n \in \mathbb{N}_0\}$.
Then $G$ contains $\mathbb{Z}$. Hence, $D$ is an almost Dedekind domain
if and only if either $q=0$ or $q \neq 0$ and $\frac{1}{q^k} \not\in G$ for some $k \in \mathbb{N}$ by Theorem \ref{gp}. It is clear that if $q \neq 0$, then $q \neq p$
if and only if $\frac{1}{q^k} \not\in G$ for some $k \in \mathbb{N}$.
Therefore, $D$ is almost Dedekind if and only if either $q = 0$ or $q \neq 0$ and $q \neq p$,
if and only if $q \neq p$.
\end{proof}

However, factorization properties of almost Dedekind domains
have been hardly investigated in terms of the multiplicative ideal theory.\ec This viewpoint is proved to be immensely useful, especially when one tries to find element-wise factorization properties in such a ring $D=\bigcup\limits_{n\in\mathbb{N}_0}F[X^{\frac{1}{p^n}}, X^{-\frac{1}{p^n}}]$ of Corollary \ref{thm31},
as we will see in the following sections. For instance, as mentioned in Section 1.1, starting from Section 3, the paper delves into the factorization properties of them such as the irreducibility of nonzero elements.

\subsection{Notations}
Let $D$ be an integral domain with quotient field $K$
and Max$(D)$ the set of maximal ideals of $D$. A $D$-submodule $A$ of $K$ is called
a \textit{fractional ideal} of $D$ if $dA \subseteq D$ for some $0 \neq d \in D$.
A fractional ideal $I$ of $D$ is called an (integral) ideal if $I \subseteq D$.
Let $I$ be a fractional ideal of $D$. Then $I^{-1} = \{x \in D \mid xI \subseteq D\}$
is also a fractional ideal of $D$,
and $I$ is said to be \textit{invertible} if $II^{-1}=D$.
A \textit{Pr\"ufer domain} is an integral domain all of whose nonzero finitely generated ideals are invertible,
so $D$ is a Pr\"ufer domain if and only if $D_M$ is a valuation domain for all $M \in$ Max$(D)$ \cite[Theorem 22.1]{gilmer}.
 It is known that $I$ is invertible if and only if
$I$ is a nonzero finitely generated ideal and $ID_M$ is principal for all $M \in$ Max$(D)$ \cite[Theorem 58, Theorem 62]{K}.
Hence, $D$ is a Dedekind domain if and only if $D$ is a Noetherian Pr\"ufer domain \cite[Theorem 37.1]{gilmer}.

Let Inv$(D)$ be the group of invertible fractional ideals of $D$ under
the usual ideal multiplication and Prin$(D)$ its subgroup of principal fractional
ideals of $D$. Then the \textit{ideal class group} or \textit{Picard group} of $D$ is the factor group
Pic$(D) =$ Inv$(D)/$Prin$(D)$ of Inv$(D)$ modulo Prin$(D)$. A \textit{B{\'e}zout domain} is an integral domain
whose finitely generated ideals are principal. For example, a PID and a valuation domain are B{\'e}zout domains.
It is clear that $D$ is a PID (resp., B{\'e}zout domain) if and only if $D$ is a Dedekind domain (resp., Pr\"ufer domain)
with Pic$(D)=\{0\}$.

Let Spec$(D)$ be the set of all prime ideals of $D$ and $|A|$ the cardinality of a set $A$.
A chain of prime ideals of $D$ is a nonempty subset $\mathcal{C}$ of Spec$(D)$ such that if $P, Q \in \mathcal{C}$,
then either $P \subseteq Q$ or $Q \subseteq P$. The length of $\mathcal{C}$
is defined by $|\mathcal{C}|-1$. The Krull dimension of $D$, denoted by dim$D$, is the largest
cardinal number $\alpha$ (if any) such that there exists a chain of prime ideals of $D$ whose length is equal to $\alpha$.
We mean by dim$D \geq \alpha$ if there is a chain of prime ideals of $D$ with length $\geq \alpha$.
The height of a prime ideal $P$ of $D$ is defined by ht$P =$ dim$D_P$. Hence, dim$D=$ sup$\{$ht$M \mid M \in$ Max$(D)\}$,
and dim$D=1$ if and only if $D$ is not a field and each nonzero prime ideal of $D$ is a maximal ideal.
Terms not defined here can be found in \cite{gilmer, K, zs60}.


\section{A system of Dedekind domains}

In this section, we will discuss the condition of the direct union of an ascending chain of Dedekind domains
being an almost Dedekind domain. We will focus on building up a series of general results,
postponing the presentation of specific examples to upcoming sections.
Some of the results in this section can be  proven by utilizing the results of \cite{arg67, bst96, ll03}; precisely, (1) and (3) of the Lemma \ref{lemma1.1} can be derived from Propositions 1.3 and 1.4 of \cite{bst96}, Lemma \ref{lemma2.3} is a consequence of Proposition 1.3 of \cite{bst96}, (1) (resp., (2), (3))  of Theorem \ref{th1} can be obtained from \cite[Lemma 3.3]{arg67} (resp., \cite[Corollary 3.6]{arg67} \cite[Theorem 2.10]{ll03}), and
(1) of Corollary \ref{coro nondiscrete} can be deduced from \cite[Theorem 3.4]{arg67}. However, for the sake of completeness, we include their proofs.

\begin{notation} \label{notation1}
{\em Let $D_1\subsetneq D_2\subsetneq D_3\subsetneq \cdots$ be a strictly increasing chain of Dedekind domains
satisfying the following three conditions for each $j \in \mathbb{N}$:
\begin{enumerate}[font=\normalfont]
\item[(i)] Each $D_j$ is not a field.
\item[(ii)] Each maximal ideal of $D_i$ survives in $D_j$ for each $i \in \mathbb{N}$ with $i<j$.
\item[(iii)] Each maximal ideal of $D_j$ contracts to a maximal ideal of $D_1$.
\end{enumerate}
Let $D=\bigcup\limits_{n \in \mathbb{N}}D_n$. Hence, $D$ is an integral domain with quotient field $\bigcup\limits_{n \in \mathbb{N}}qf(D_n)$,
where $qf(D_n)$ is the quotient field of $D_n$ for each $n \in \mathbb{N}$.
In this case, we will call $\{D_n\}_{n\in\mathbb{N}}$ a \textit{system of Dedekind domains} and $D$ the \textit{limit} of the system.}
\end{notation}

It is very helpful to note that if $\{D_n\}_{n\in\mathbb{N}}$ is a system of Dedekind domains
and $k \in \mathbb{N}$, then $\{D_{k+n}\}_{n\in\mathbb{N}_0}$ is also a system of Dedekind domains
such that $\bigcup\limits_{n \in \mathbb{N}_0}D_{n+k} = \bigcup\limits_{n \in \mathbb{N}}D_n.$
Firstly, we establish the maximal ideal structure of the limit of a system of Dedekind domains in the following two lemmas.

\begin{lemma} \label{lemma1.1}
Let $\{D_n\}_{n\in\mathbb{N}}$ be a system of Dedekind domains and $D=\bigcup\limits_{n\in\mathbb{N}}D_n$.
\begin{enumerate}[font=\normalfont]
\item Given $j\in\mathbb{N}$ and a maximal ideal $M_j$ of $D_j$, there exists a maximal ideal $M$ of $D$ such that $M_j=M\cap D_j$.
\item $I\cap D_1\neq(0)$ for each nonzero ideal $I$ of $D$.
In particular, a maximal ideal of $D$ contracts to a maximal ideal of $D_j$ for each $j\in\mathbb{N}$.
\item For each $p\in D_n$, with $n \in \mathbb{N}$, $p$ is a prime element of $D$
if and only if $p$ is a prime element of $D_{n+k}$ for all $k \in \mathbb{N}_0$.
\item $qf(D_n)\cap D=D_n$ for any $n\in\mathbb{N}$.
\end{enumerate}
\end{lemma}

\begin{proof}
(1) Suppose that $M_jD=D$. Then $1=d_1m_1+\cdots d_nm_n$ for some $d_i\in D$ and $m_i\in M_j$.
Choose $k\in\mathbb{N}$ such that $d_i\in D_k$ for all $i\in\{1,\dots, n\}$ and let $t=\max\{k,j\}$.
Then we have $M_jD_t=D_t$, which contradicts our assumption. Hence, $M_j$ survives in $D$,
so $M_j$ is contained in a maximal ideal $M$ of $D$.
It follows that $M_j\subseteq M\cap D_j$ is a chain of nonzero prime ideals of a Dedekind domain $D_j$, so $M_j=M\cap D_j$.

(2) Choose $i\in\mathbb{N}$ such that $I\cap D_i\neq(0)$. Then $I\cap D_i=M_1^{a_1}\cdots M_n^{a_n}=M_1^{a_1}\cap\cdots\cap M_n^{a_n}$
for some maximal ideals $M_1,\dots, M_n$ of $D_i$ and positive integers $a_1,\dots, a_n$.
It follows that for each $r\in\{1,\dots, n\}$, $M_r^{a_r}\cap D_1$ is nonzero, for it contains $(M_r\cap D_1)^{a_r}$,
which is nonzero by assumption. Hence, $I\cap D_1 \neq (0)$.

(3) Let $p\in D_n$ and choose $k\ge n$. If $p$ is a prime element of $D$, then $pD_k\subseteq pD\cap D_k$.
 Since $D_k$ is a Dedekind domain, we have $pD_k=(pD\cap D_k)^aI$ for some $a\in\mathbb{N}$ and an ideal $I$ of $D_k$ with $(pD \cap D_k) +I = D_k$.
 Then $pD=pD_kD=(pD\cap D_k)^aID\subseteq (pD)^a(ID) \subseteq (pD)^a \cap ID$, which implies that $a=1$ and $I=D_k$ by (1).
 Thus, $p$ is a prime element of $D_k$. Conversely, assume that $p$ is a prime element of $D_k$ for every $k\ge n$.
 Then $pD=\bigcup\limits_{k\ge n}pD_k$, and since $pD_k$ is a prime ideal of $D_k$
 for all $k \geq n$, we have that $pD$ is a prime ideal.

(4) Let $T_m=qf(D_n)\cap D_{m}$ for each $m> n$.
Then $T_m$ is an overring of $D_n$ and
since $D_n$ is a Pr{\"u}fer domain,
$Q= (Q \cap D_n)T_m$ for each prime ideal $Q$ of $T_m$ \cite[Theorem 26.1]{gilmer}.
Thus, the contraction map $\textnormal{Spec}(T_m)\to\textnormal{Spec}(D_n)$ is an injective map.
On the other hand, by the assumption on the chain $D_1\subsetneq D_2\subsetneq \cdots$,
the contraction map $\textnormal{Spec}(D_{m})\to\textnormal{Spec}(D_n)$ is a surjective map.
Hence, the composition of the contraction maps $\textnormal{Spec}(D_{m})\to\textnormal{Spec}(T_m)\to\textnormal{Spec}(D_n)$ is surjective
by which we deduce that $\textnormal{Spec}(T_m)\to \textnormal{Spec}(D_n)$ is a bijective map.
Since $T_m=\bigcap\limits_{P}(D_n)_P$ where the intersection is taken over all prime ideals $P$ of $D_n$ such that $PT_m\neq T_m$ \cite[Theorem 26.1]{gilmer}, we must have $T_m=D_n$.
Finally, $$qf(D_n)\cap D=\bigcup\limits_{m> n}(qf(D_n)\cap D_m)=\bigcup\limits_{m> n}T_m=D_n,$$
so the desired conclusion follows.
\end{proof}

\begin{lemma} \label{lemma2.3}
Let $\{D_n\}_{n\in\mathbb{N}}$ be a system of Dedekind domains, $D=\bigcup\limits_{n\in\mathbb{N}}D_n$,
$M$ a maximal ideal of $D$ and $M_i=M\cap D_i$ for each $i\in\mathbb{N}$.
\begin{enumerate}[font=\normalfont]
\item $M=\bigcup\limits_{i\in\mathbb{N}}M_i$.
\item $D_M=\bigcup\limits_{i\in\mathbb{N}}(D_i)_{M_i}$.
\item $MD_M=\bigcup\limits_{i\in\mathbb{N}}M_i(D_i)_{M_i}$.
\item $\{(D_i)_{M_i}\}_{i\in\mathbb{N}}$ is a system of Dedekind domains.
\item Two maximal ideals $M$ and $M'$ coincide if and only if $M\cap D_i=M'\cap D_i$  for each $i\in\mathbb{N}$.
\item $D/M = \bigcup\limits_{i \in \mathbb{N}}D_i/M_i$.
Hence, $|D/M| < \infty$ if and only if $|D_n/M_n| < \infty$ for some $n \in \mathbb{N}$
and $D_i/M_i = D_n/M_n$ for all $i \in \mathbb{N}$ with $i \geq n$.
\end{enumerate}
\end{lemma}

\begin{proof}
(1) This is trivial.

(2) Note that $M_i$ is a maximal ideal of $D_i$ for each $i \in \mathbb{N}$ by Lemma \ref{lemma1.1}(2).
Moreover, $(D_i)_{M_i}$ is a subring of $D_M$ for each $i\in\mathbb{N}$ since $D_i\subseteq D$ and $M \cap D_i = M_i$.
Hence, $\bigcup\limits_{i\in\mathbb{N}}(D_i)_{M_i} \subseteq D_N$.
For reverse containment, let $x\in D_M$, so $x=\frac{d}{c}$ for some $d\in D$ and $c \in D\setminus M$.
Then $d\in D_i$, $c \in D_j$ and $c \not\in M_k$ for some $i,j,k\in\mathbb{N}$, so if $l=\max\{i,j,k\}$,
we have $x\in (D_l)_{M_l}$. Thus, $D_M \subseteq \bigcup\limits_{i\in\mathbb{N}}(D_i)_{M_i}$.

(3) Following the argument in the proof of (2) above, we have $MD_M=\bigcup\limits_{i\in\mathbb{N}}M_i(D_i)_{M_i}$.

(4) This follows directly from (2) and (3) above.

(5) This is an immediate consequence of (1) above.

(6) It is clear that $D/M_i$ is a subfield of $D/M$ for each $i\in\mathbb{N}$.
Hence, $\bigcup\limits_{i \in \mathbb{N}}D_i/M_i \subseteq D/M$.
Conversely, let $y \in D$. Then $y \in D_i$ for some $i \in \mathbb{N}$, and
hence $y+M = y+M_i \in D_i/M_i$. Thus, $D/M \subseteq \bigcup\limits_{i \in \mathbb{N}}D_i/M_i$.
\end{proof}

We next study some ring-theoretic properties of the limit of a system of Dedekind domains
and characterize when it is an almost Dedekind domain.  An integral domain
in which each nonzero nonunit is contained in only finitely many maximal ideals is said to be of \textit{finite character}.
It is known that an integral domain is Dedekind if and only if it is an almost Dedekind domain of finite character \cite[Theorem 37.2]{gilmer}.

\begin{theorem} \label{th1}
Let $\{D_n\}_{n\in\mathbb{N}}$ be a system of Dedekind domains and $D=\bigcup\limits_{n\in\mathbb{N}}D_n$.
\begin{enumerate}[font=\normalfont]
\item $D$ is a Pr\"ufer domain of Krull dimension 1.
In particular, if $M_i$ is a maximal ideal of $D_i$ and $M_{i}\subseteq M_{i+1}$ for each $i\in\mathbb{N}$,
then $M=\bigcup\limits_{i\in\mathbb{N}}M_i$ is a maximal ideal of $D$.
\item  $D$ is an almost Dedekind domain if and only if there exists an $n \in\mathbb{N}$ such that
$(M'\cap D_n)(D_{j})_{M'}=M'(D_{j})_{M'}$ for each  $j \ge n$ and a maximal ideal $M'$ of $D_j$.
\item $D$ is a Dedekind domain if and only if $D$ is an almost Dedekind domain and each maximal ideal of $D_1$ is
contained in only finitely many maximal ideals of $D$.
\end{enumerate}
\end{theorem}

\begin{proof}
(1) Suppose that $I=(a_1,\dots, a_m)$ is a nonzero finitely generated ideal of $D$.
Then there exists $n\in\mathbb{N}$ such that $a_i\in D_n$ for each $i\in\{1,\dots, m\}$.
Thus, if we let $J = (a_1, \dots , a_m)D_n$, then $J$ is an invertible ideal
of $D_n$ and $(a_1,\dots, a_m)D_k=JD_k$ for every $k \ge n$.
Therefore, $I=JD$ is an invertible ideal of $D$. Hence, $D$ is a Pr\"ufer domain.
To show that the Krull dimension of $D$ is 1, note that if $P\subsetneq M$ is a chain of prime ideals of $D$,
then there exists $n\in\mathbb{N}$ such that $P\cap D_n\subsetneq M\cap D_n$.
Indeed, choose $a\in M\setminus P$ and $n\in\mathbb{N}$ such that $a\in D_n$.
Now if $k$ is an integer with $k\ge n$, then $P\cap D_k\subsetneq M\cap D_k$, and since $D_k$ is a Dedekind domain,
we must have $P\cap D_k=(0)$. Since $k$ is an arbitrarily chosen integer greater than $n$,
we conclude that $P$ is the zero ideal of $D$. Hence, $D$ has Krull dimension one.

(2) By Lemma \ref{lemma1.1}(3), we may assume that $D, D_1, D_2,\dots$ are local rings.
Let $M_j$ be the maximal ideal of $D_j$ for each $j\in\mathbb{N}$. Suppose that $D$ is a DVR with maximal ideal $M$.
Then $M_j=M\cap D_j$ by Lemma \ref{lemma1.1}(1). Since $M_j$ survives in $D_{j+1}$,
we have that $M_jD_{j+1}=M_{j+1}^{k_j}$ for some $k_j\in\mathbb{N}$. Note that $k_j>1$ for only finitely many $j\in\mathbb{N}$,
since $M_1\subseteq M^{\prod\limits_{i=1}^rk_j}$ for each $r\in\mathbb{N}$ and $\bigcap\limits_{i=1}^{\infty}M^i=(0)$.
Choose the smallest $l\in\mathbb{N}$ such that $k_j=1$ for each $j\ge l$.
Then $M_lD_{j}=M_j$ and $M_{j}\cap D_l=M_l$ for $j \geq l$. Thus,
$(M_{j}\cap D_l)D_{j}=M_{j}$ for each integer $j\ge l$.

Conversely, suppose that there exists $l\in\mathbb{N}$ such that for each integer $j\ge l$ and a maximal ideal $M_j$ of $D_j$,
we have $(M_j\cap D_l)D_{j}=M_j$. Since $D_l$ is a PID, there exists a generator $x$ of $M_l$ as an ideal of $D_l$.
Then $x$ is a generator of $M_j$ as an ideal of $D_j$ for every $j\ge
 l$, since $M_j\cap D_l=M_l$ by Lemma \ref{lemma1.1}(1).
 Hence, it is a generator of $M$ as an ideal of $D$. Thus, $D$ is a DVR.

(3) Suppose that $D$ is Dedekind. Then it is an almost Dedekind domain
 and each maximal ideal of $D_1$ survives in $D$ by Lemma \ref{lemma1.1}(1),
so it is contained in only finitely many maximal ideals of $D$. Conversely, suppose that $D$ is an almost Dedekind domain
and each maximal ideal of $D_1$ is contained in only finitely many maximal ideals of $D$. Choose a nonzero nonunit $a\in D$.
Since $aD\cap D_1$ is contained in finitely many maximal ideals of $D_1$, it is contained in only finitely many maximal ideals of $D$.
Since every maximal ideal of $D$ that contains $a$ also contains $aD\cap D_1$,
there must be only finitely many maximal ideals of $D$ that contains $a$. It follows that $D$ is an almost Dedekind domain of finite character.
Hence, $D$ is a Dedekind domain.
\end{proof}

\begin{corollary} \label{coro2.4}
Let $\{D_n\}_{n\in\mathbb{N}}$ be a system of Dedekind domains, $D=\bigcup\limits_{n\in\mathbb{N}}D_n$
and $M$ a maximal ideal of $D$.
Then
$M$ is either invertible or countably infinitely generated.
\end{corollary}

\begin{proof}
By Lemma \ref{lemma2.3}(1), $M = \bigcup\limits_{i \in \mathbb{N}}(M \cap D_i)$, so
if $M \cap D_i = (q_{i1}, \cdots , q_{ik_i})D_i$ for some $q_{ij} \in M \cap D_i$, then
$M= (\{q_{i1}, \cdots , q_{ik_i} \mid i \in \mathbb{N}\})$. Hence, $M$ is countably generated.
Now, by Theorem \ref{th1}, $D$ is a Pr\"ufer domain, so if $M$ is finitely generated,
then $M$ is invertible. Otherwise, $M$ is countably infinitely generated.
\end{proof}

\begin{corollary} \label{coro2.5}
Let $\{D_n\}_{n\in\mathbb{N}}$ be a system of Dedekind domains,
$D=\bigcup\limits_{n\in\mathbb{N}}D_n$ and $M$ a maximal ideal of $D$. Then
$M$ is invertible if and only if
there is $n \in \mathbb{N}$ such that $(M \cap D_n){D_j} = M \cap D_j$
for all $j \in \mathbb{N}$ with $j \geq n$. In this case, $M = (M \cap D_n)D$.
\end{corollary}

\begin{proof}
By Theorem \ref{th1}(1), $M = \bigcup\limits_{i \in \mathbb{N}}(M \cap D_i)$. Hence, if
$(M \cap D_n){D_j} = M \cap D_j$ for all $j \in \mathbb{N}$ with $j \geq n$,
then $M = (M \cap D_n)D$. Note that $M \cap D_n$ is a finitely generated ideal of $D_n$, so
$M$ is finitely generated. Thus, $M$ is invertible because $D$ is a Pr\"ufer domain
by Theorem \ref{th1}(1).
Conversely, if $M$ is invertible, then $M$ is finitely generated,
and hence $M = (M \cap D_n)D$ for some  $n \in \mathbb{N}$. Thus,
$(M \cap D_n){D_j} = M \cap D_j$ for all $j \in \mathbb{N}$ with $j \geq n$.
\end{proof}

If $D$ is the limit of a system of Dedekind domains, then $D$ is a one-dimensional Pr\"ufer domain by Theorem \ref{th1}(1),
so $D_M$ is a rank one valuation domain for each maximal ideal $M$ of $D$; in particular,
$D_M$ is a DVR if $M$ is invertible.
In the following corollary, we determine when $D_M$ is a DVR, extending Theorem \ref{bytheo}.

\begin{corollary} \label{coro nondiscrete}
Let $\{D_n\}_{n\in\mathbb{N}}$ be a system of Dedekind domains, $D=\bigcup\limits_{n\in\mathbb{N}}D_n$
and $M$ a maximal ideal of $D$. Then
the following statements are equivalent.
\begin{enumerate}[font=\normalfont]
\item $D_M$ is a DVR.
\item There exists $n \in \mathbb{N}$ such that $$(M \cap D_n){(D_j)}_{M \cap D_j} = (M \cap D_j){(D_j)}_{M \cap D_j}$$
for all $j \in \mathbb{N}$ with $j \geq n$.
\item There exists $n \in \mathbb{N}$ such that $M \cap D_n \nsubseteq (M \cap D_j)^2$
for all $j \in \mathbb{N}$ with $j \geq n$.
\end{enumerate}
In this case, $MD_M = (M \cap D_n)D_M$.
\end{corollary}

\begin{proof}
(1) $\Leftrightarrow$ (2) This is already proved in the proof of Theorem \ref{th1}(2).

(2) $\Leftrightarrow$ (3) It suffices to show that
$(M \cap D_n){(D_j)}_{M \cap D_j} \neq (M \cap D_j){(D_j)}_{M \cap D_j}$ if and only if
$M \cap D_n \subseteq (M \cap D_j)^2$.
Assume that $$(M \cap D_n){(D_j)}_{M \cap D_j} \neq (M \cap D_j){(D_j)}_{M \cap D_j}.$$
Then $(M \cap D_n){(D_j)}_{M \cap D_j} \subseteq ((M \cap D_j){(D_j)}_{M \cap D_j})^2
= (M \cap D_j)^2{(D_j)}_{M \cap D_j},$
because ${(D_j)}_{M \cap D_j}$ is a DVR, and since $(M \cap D_j)^2$ is a primary ideal of $D_j$,
$M \cap D_n \subseteq (M \cap D_n){(D_j)}_{M \cap D_j} \cap D_j \subseteq (M \cap D_j)^2{(D_j)}_{M \cap D_j} \cap D_j
= (M \cap D_j)^2$. Thus, $M \cap D_n \subseteq (M \cap D_j)^2$.
Conversely, assume that $M \cap D_n \subseteq (M \cap D_j)^2$.
Then $(M \cap D_n){(D_j)}_{M \cap D_j} \subseteq (M \cap D_j)^2{(D_j)}_{M \cap D_j}
= ((M \cap D_j){(D_j)}_{M \cap D_j})^2 \subsetneq (M \cap D_j){(D_j)}_{M \cap D_j}$.
Thus, $(M \cap D_n){(D_j)}_{M \cap D_j} \neq (M \cap D_j){(D_j)}_{M \cap D_j}$.
\end{proof}

Let $T_1\subseteq T_2$ be an extension of Dedekind domains and $P$ a prime ideal of $T_1$.
Following \cite[p.102]{milne}, we will say that $P$ is \textit{inert} in $T_2$ if $PT_2$ is a prime ideal of $T_2$.
If every maximal ideal of $T_1$ is inert in $T_2$, then we will say that the extension $T_1\subseteq T_2$ is \textit{inert}.
This definition can be used to determine
when a system of Dedekind domains $\{D_n\}_{n\in\mathbb{N}}$ yields a Dedekind domain $D=\bigcup\limits_{n\in\mathbb{N}}D_n$.

\begin{corollary} \label{dede}
Let $\{D_n\}_{n\in\mathbb{N}}$ be a system of Dedekind domains
 and $D=\bigcup\limits_{n\in\mathbb{N}}D_n$. Then the following
statements are equivalent.
\begin{enumerate}[font=\normalfont]
    \item $D$ is a Dedekind domain.
     \item For each $n\in\mathbb{N}$ and a prime ideal $P$ of $D_n$,
     there exists $l\ge n$ so that for each $m\ge l$, a prime ideal of $D_m$ that contains $P$ is inert in $D_{m+1}$.
\end{enumerate}
\end{corollary}

\begin{proof}
(1) $\Rightarrow$ (2) Recall that an integral domain is Dedekind if and only if it is an almost Dedekind domain of finite character.
If $D$ is almost Dedekind, then given a prime ideal $P$ of $D_n$,
there exists $k\in\mathbb{N}$, with $k \geq n$, such that for every $j\ge k$ and a prime ideal $Q$ of $D_j$ that contains $P$,
$Q$ is not contained in $M^2$ for each prime ideal $M$ of $D_{j+1}$ by Corollary \ref{coro nondiscrete}.
On the other hand, if $D$ is of finite character, then there exists $l \in\mathbb{N}$, with $l \geq n$,
such that each prime ideal $H$ of $D_m$
that contains $P$ is contained in a single maximal ideal of $D_{m+1}$ for each $m \ge l$ by Lemma \ref{lemma2.3}(5) and Theorem \ref{th1}(3).
Thus, if $D$ is a Dedekind domain, then, for each prime ideal $P$ of $D_n$,
there exists $l\in\mathbb{N}$, with $l \geq n$, such that for each $m\ge l$ and a prime ideal $H$ of $D_m$ that contains $P$, $H$ is inert in $D_{m+1}$.

(2) $\Rightarrow$ (1) Let $M$ be a maximal ideal of $D$. Then
$(M \cap D_l)D_j = M \cap D_j \nsubseteq (M \cap D_j)^2$ for all $j \geq l$ by (2).
Thus, $M \cap D_l \not\subseteq (M \cap D_j)^2$ for all $j \geq l$, and $D_M$ is a DVR by Corollary \ref{coro nondiscrete}.
Moreover, if $P$ is a maximal ideal of $D_1$, then $P$ is contained in
only finitely many maximal ideals of $D_l$ and each of them is contained in a
unique maximal ideal of $D$ by (2) and Lemma \ref{lemma2.3}(5). Hence, $P$ is contained in only finitely
many maximal ideals of $D$. Thus, $D$ is a Dedekind domain by Theorem \ref{th1}(3).
\end{proof}

We now turn our attention to the height of a prime ideal of a power series ring over
the limit of a system of Dedekind domains. Let $E[[X]]$ be the power series ring over an integral domain $E$.
For a nonzero ideal $I$ of $E$, let $I[[X]] = \{\sum\limits_{i=0}^{\infty}a_iX^i \mid a_i \in I\textnormal{ for all }i \geq 0\}.$
Then the following statements hold:
\begin{itemize}
\item $I[[X]]$ is an ideal of $E[[X]]$ such that $IE[[X]] \subseteq I[[X]]$, and equality holds if $I$ is finitely generated.
\item $I[[X]]$ is a prime ideal if and only if $I$ is a prime ideal.
\item $E[[X]]/I[[X]] \cong (E/I)[[X]]$.
\end{itemize}
\noindent
An ideal $I$ of $E$ is called an \textit{SFT-ideal} (\textit{strong finite type}) if
there exist a finitely generated ideal $J$ of $E$ and  $k \in \mathbb{N}$
such that $a^k \in I$ for all $a \in I$.
Then $E$ is called an \textit{SFT-ring} if each ideal of $E$
is an SFT-ideal or, equivalently, if each prime ideal of $E$ is an SFT-ideal \cite[Proposition 2.2]{ar2}.
It is known that dim$(E[[X]]) = \infty$ if $E$ is not an SFT-ring \cite[Theorem 1]{ar}.
With these results in mind, one can then proceed to the following.

\begin{corollary} \label{coro2.6}
Let $\{D_n\}_{n\in\mathbb{N}}$ be a system of Dedekind domains,
$D=\bigcup\limits_{n\in\mathbb{N}}D_n$ and
$M$ a maximal ideal of $D$ such that $D_M$ is a DVR.
\begin{enumerate}[font=\normalfont]
\item The following statements are equivalent.
\begin{enumerate}[font=\normalfont]
\item $M$ is finitely generated.
\item $M$ is invertible.
\item ht$(M[[X]]) =1$.
\item ht$(M[[X]]) < \infty$.
\end{enumerate}
\item $M$ is infinitely generated if and only if ht$(M[[X]]/MD[[X]]) \geq 2^{\aleph_1}$.
\end{enumerate}
\end{corollary}

\begin{proof}
(1) (a) $\Leftrightarrow$ (b) This follows because $D$ is a Pr\"ufer domain by Theorem \ref{th1}(1).
(b) $\Rightarrow$ (c) An invertible ideal is locally principal, so
$MD_M = qD_M$ for some $q \in D$.
Then, as $M$ is finitely generated, $MD[[X]]_M = M[[X]]_M = qD[[X]]_M$.
Now let $Q$ be a prime ideal
of $D[[X]]$ such that $Q_M \subsetneq qD[[X]]_M$.
If $f = \sum\limits_{i = 0}^{\infty}a_iX^i \in Q_M$,
then $f = qf_1$ for some $f_1 \in D[[X]]_M$. Since $q \not\in Q_M$, we have
$f_1 \in Q_M$. By induction, we have an $f_m \in D[[X]]_M$ such that
$f = q^mf_m$ for all $m \geq 1$. This shows that
$a_i \in \bigcap\limits_{m=1}^{\infty}q^mD_M = (0)$ for all $i \in \mathbb{N}_0$. Hence, $f = 0$,
and so $Q = Q_M \cap D[[X]] = (0)$. Thus, ht$(M[[X]]) =$ ht$(M[[X]]_M)=$ ht$(qD[[X]]_M) = 1$. (c) $\Rightarrow$ (d) Clear.
(d) $\Rightarrow$ (b) Assume to the contrary that $M$ is not invertible.
Then $M$ is not an SFT-ideal by the proof of \cite[Lemma 2.1]{ckt15}
and $M$ is countably infinitely generated by Corollary \ref{coro2.4}.
Thus, ht$(M[[X]]) \geq$ ht$(M[[X]]/MD[[X]]) \geq 2^{\aleph_1}$
\cite[Theorem 15]{tk21}.

(2) If $M$ is infinitely generated, then ht$(M[[X]]/MD[[X]]) \geq 2^{\aleph_1}$ by the proof of (d) $\Rightarrow$ (b) above.
Conversely, assume ht$(M[[X]]/MD[[X]]) \geq 2^{\aleph_1}$. Then ht$(M[[X]])$ = $\infty$, and hence
$M$ is infinitely generated by (1) above.
\end{proof}

It is very well known that if $E[X]$ is the polynomial ring
over a Pr\"ufer domain $E$, then dim$(E[X]) =$ dim$(E) +1$ \cite[Theorem 4]{s53}.
Hence, if $E$ is an almost Dedekind domain that is not a field, then dim$(E[X])=2$
and ht$(ME[X]) = 1$ for all $M \in$ Max$(E)$. However, dim$(E[[X]])$
behaves quite differently as we can see in the following corollary
(see \cite{tk18}).

\begin{corollary} \label{coro2.7}
Let $\{D_n\}_{n\in\mathbb{N}}$ be a system of Dedekind domains and
$D=\bigcup\limits_{n\in\mathbb{N}}D_n$. Then the following statements are
satisfied.
\begin{enumerate}[font=\normalfont]
\item $D$ is a Dedekind domain if and only if dim$(D[[X]]) = 2$.
\item $D$ is not a Dedekind domain if and only if dim$(D[[X]]) \geq 2^{\aleph_1}$,
and equality holds under the continuum hypothesis if $D$ is countable.
\item Let $M$ be a maximal ideal of $D$, and assume that $D$ is an almost Dedekind domain. Then
either ht$(M[[X]]) = 1$ or ht$(M[[X]]) \geq 2^{\aleph_1}$.
\end{enumerate}
\end{corollary}

\begin{proof}
An integral domain is a Dedekind domain if and only if it is a one-dimensional SFT Pr\"ufer domain
(\cite[Corollary 5.4.11]{fhp97} and \cite[Theorem 2.4]{kp99}). Recall from
Theorem \ref{th1}(1) that $D$ is a one-dimensional Pr\"ufer domain. Thus, $D$ is a Dedekind domain
if and only if $D$ is a SFT-ring.

(1) If $D$ is a Dedekind domain, then $D$ is a one-dimensional Noetherian domain, and hence
dim$(D[[X]])=2$ \cite[Theorem 19]{a81}. Conversely, assume that dim$(D[[X]])=2$. Then $D$ is an
SFT-ring \cite[Theorem 20]{a81}.

(2) If $D$ is not a Dedekind domain, then  dim$(D[[X]]) \geq 2^{\aleph_1}$ \cite[Theorem 3.9]{tk18}.
The converse is from (1) above. Moreover, if $D$ is countable,
then  dim$(D[[X]]) = 2^{\aleph_1}$ under the continuum hypothesis \cite[Corollary 3.11]{tk18}.

(3) $D_M$ is a DVR since $D$ is an almost Dedekind domain by assumption.
Thus, the result is a consequence of Corollary \ref{coro2.6}.
\end{proof}

A PID is a Dedekind domain, so we can consider a special type of a system of Dedekind domains,
say, $\{D_n\}_{n \in \mathbb{N}}$ such that $D_n$ is a PID for all $n \in \mathbb{N}$.
We briefly summarize the basic properties of this type in the next corollary,
and continue the investigation in the forthcoming sections.

\begin{corollary} \label{pid}
Let $\{D_n\}_{n\in\mathbb{N}}$ be a system of Dedekind domains,
$D=\bigcup\limits_{n\in\mathbb{N}}D_n$, and assume that $D_n$ is a PID for all $n \in \mathbb{N}$. Then the following
statements hold.
\begin{enumerate}[font=\normalfont]
\item $D$ is a B{\'e}zout domain of Krull dimension one.
\item $D$ is a PID if and only if $D$ is an almost Dedekind domain and each maximal ideal of $D_1$ is
contained in only finitely many maximal ideals of $D$.
\item Let $n, m \in \mathbb{N}$ with $n <m$ and $a, b \in D_n$ be nonzero nonunits.
If $a|b$ in $D_m$, then $a|b$ in $D_n$.
\item For a nonzero nonunit $a \in D$, let $P(a)$ be the set of pairwise nonassociated prime elements of $D$
dividing $a$ in $D$. Then $P(a)$ is countable.
\end{enumerate}
\end{corollary}

\begin{proof}
(1) By Theorem \ref{th1}(1), $D$ is a Pr\"ufer domain of Krull dimension one.
Also, if $I$ is a nonzero finitely generated ideal of $D$, then the proof of
Theorem \ref{th1}(1) shows that $I$ is principal, because each $D_n$ is a PID.


(2) A Dedekind domain that is a B{\'e}zout domain is a PID. Thus,
the result follows directly from Theorem \ref{th1}(3) because $D$ is a B{\'e}zout domain.

(3) This is an immediate consequence of Lemma \ref{lemma1.1}(4).

(4) For each $n \in \mathbb{N}$, let $S_n = P(a) \cap D_n$. Then $S_n \subseteq S_{n+1}$
and $P(a) = \bigcup\limits_{n \in \mathbb{N}}S_n$. Moreover, $|S_n| < \infty$
since $qf(D_n) \cap D = D_n$ by Lemma \ref{lemma1.1}(4). Thus, $P(a)$ is countable.
\end{proof}

The next example gives a system of Dedekind domain which shows
that Corollary \ref{pid}(1) may hold even if $D_n$ is not a PID for every $n\in\mathbb{N}$.

\begin{example}
{\em Let $\overline{\mathbb{Q}}$ be the set of algebraic numbers.
Then $\overline{\mathbb{Q}}$ is countable, so there exists a field extension $F_n\subsetneq F_{n+1}$ for each $n\in\mathbb{N}$ such that  $F_1=\mathbb{Q}$,
$[F_{n+1}:F_n] < \infty$ for all $n \in \mathbb{N}$ and $\bigcup\limits_{n\in\mathbb{N}}F_n=\overline{\mathbb{Q}}$.
Now let $D_n$ be the integral closure of $\mathbb{Z}$ in $F_n$ for each $n\in\mathbb{N}$.
It is well-known that $D_n$ is a Dedekind domain, $Pic(D_n)$ is a finite group for each $n\in\mathbb{N}$
but $D_n$ is not necessarily a PID. Moreover, $\{D_n\}_{n \in \mathbb{N}}$ is a system of Dedekind domains
and $D=\bigcup\limits_{n\in\mathbb{N}}D_n$, the ring of algebraic integers, is a B{\'e}zout domain of Krull dimension one \cite[Theorem 102]{K}. }
\end{example}

We end this section with a couple of simple examples of integral domains which can be
constructed from a strictly ascending chain of PIDs satisfying the condition
of Notation \ref{notation1}.

\begin{example}
\label{ex2.12}
{\em Let $\{F_i\}_{i\in\mathbb{N}}$ be a strictly ascending chain of fields, $F=\bigcup\limits_{i\in\mathbb{N}}F_i$,
$X$ an indeterminate over $F$, $F_i[X]$ the polynomial ring over $F_i$ and $F_i[[X]]$ the
power series ring over $F_i$ for each $i \in \mathbb{N}$. Now let
$$D = \bigcup\limits_{i\in\mathbb{N}}F_i[X] \ \ \text{ and } \ \ R=\bigcup\limits_{i\in\mathbb{N}}F_i[[X]].$$
Then $D = F[X]$ and $R \subsetneq F[[X]]$. Moreover, the following statements hold.
\begin{enumerate}[font=\normalfont]
\item $F_{i+1}$ is algebraic over $F_i$ for each $i\in\mathbb{N}$ if and only if $\{F_i[X] \mid i \in\mathbb{N}\}$
is a system of Dedekind domains.
Moreover, $D$ is a PID.
\item $\{F_i[[X]] \mid i \in\mathbb{N}\}$ is a system of Dedekind domains.
Hence, $R$ is a B{\'e}zout domain with Krull dimension 1. In fact, $R$ is a DVR.
\item Suppose that $\{F_i[X] \mid i \in\mathbb{N}\}$ is a system of Dedekind domains.
Let $m,n\in\mathbb{N}$ be such that $n<m$ and $f,g\in F_n[X]$.
If $f$ divides $g$ in $F_m[X]$, then $f$ divides $g$ in $F_n[X]$.
\item Let $F$ be an algebraic closure of $F_1$. If $F_1$ is finite, then
there is a strictly ascending chain $\{F_i\}_{i\in\mathbb{N}}$ of fields such that $F=\bigcup\limits_{i\in\mathbb{N}}F_i$.
\end{enumerate}}
\end{example}

\begin{proof}
(1) is clear. For (2), choose $a_0 \in F_1$ and $a_i \in F_{i+1} \setminus F_i$ for each $i \in \mathbb{N}$,
and let $f = \sum\limits_{i \in \mathbb{N}_0}a_iX^i$. Then $f \in F[[X]] \setminus R$, and thus $R \subsetneq F[[X]]$.
Moreover, $R$ is a B{\'e}zout domain of Krull dimension one by Corollary \ref{pid}(1)
and $M: = \bigcup\limits_{i \in \mathbb{N}}XF_i[[X]]$ is a unique maximal ideal of $R$.
Then $M$ is generated by $X$, and thus $R$ is a  DVR.
The other parts follow from the definition and Corollary \ref{pid}.
(3) is an immediate consequence of Corollary \ref{pid}(4). For (4), it is enough to
note that $F$ is countable and $[F:F_1] = \infty$.
\end{proof}

Let $E[X]$ be the polynomial ring over an integral domain $E$. For a nonzero polynomial $f \in E[X]$,
let $c(f)$ denote the ideal of $E$ generated by the coefficients of $f$. Then
$S: = \{f \in E[X] \mid c(f)=E\}$ is a saturated multiplicative subset of $E[X]$,
and hence $E[X]_S$ is an overring of $E[X]$. The ring $E[X]_S$,
denoted by $E(X)$, is called the Nagata ring of $E$. It is known that
Max$(E(X)) = \{ME(X) \mid M \in$ Max$(E)\}$ \cite[Proposition 33.1]{gilmer}; $E$ is a Pr\"ufer domain
if and only if $E(X)$ is a B{\'e}zout domain \cite[Theorem 33.4]{gilmer};
$E$ is a Dedekind domain if and only if $E(X)$ is a PID \cite[Proposition 38.7]{gilmer};
$E$ is an almost Dedekind domain if and only if $E(X)$ is an almost Dedekind domain \cite[Proposition 36.7]{gilmer};
and $E(X)$ is a field if and only if $E$ is a field.

\begin{example} \label{ex2.14}
{\em (1) Let $\{n_i \mid i \in \mathbb{N}\}$ be a set of positive integers with $n_i|n_{i+1}$ for all $i \in \mathbb{N}$,
 $D_i = \mathbb{Z}(X^{\frac{1}{n_i}})$ the Nagata ring of $\mathbb{Z}$
 and $\mathbb{P}$ the set of prime numbers. Then $D_1 \subsetneq D_2 \subsetneq \cdots$
is a strictly increasing chain of PIDs satisfying the condition of Notation \ref{notation1}
and Spec$(D_i) = \{pD_i \mid p \in \mathbb{P}\} \cup \{0\}$ for all $i \in \mathbb{N}$.
It is easy to see that each prime ideal of $D_m$ is inert in $D_{m+1}$ for all
$m \in \mathbb{N}$. Hence, $D$ is a Dedekind domain by Corollary \ref{dede}, and
since $D$ is a B{\'e}zout domain by Corollary \ref{pid}, $D$ is a PID.
Moreover, Spec$(D) = \{pD \mid p \in \mathbb{P}\} \cup \{0\}$ and $D/pD = \bigcup\limits_{i \in \mathbb{N}}(\mathbb{Z}/p\mathbb{Z})(X^{\frac{1}{n_i}})$
for all $p \in \mathbb{P}$.

(2) Let $F$ be a field, $X$ an indeterminate over $F$, $A_n = F[X^{\frac{1}{2^n}}]$ for all
$n \in \mathbb{N}_0$, $A = \bigcup\limits_{n \in \mathbb{N}_0}A_n$ and $T= F[X] \setminus (X-1)F[X]$.
Then $A$ is a B{\'e}zout domain of Krull dimension one but $A$ is not an almost Dedekind domain \cite[Proposition 3.8]{ckt15}.
In particular, if $F = \mathbb{Q}$,
then $A_T$ is an almost Dedekind domain but not a Dedekind domain \cite[Theorem 3.7]{ckt15}.
Now let $p$ be a prime number, $E_n = F[X^{\frac{1}{p^n}}, X^{- \frac{1}{p^n}}]$ for all
$n \in \mathbb{N}_0$ and $E = \bigcup\limits_{n \in \mathbb{N}_0}E_n$.
Then $E$ is a B{\'e}zout domain of Krull dimension one by Corollary \ref{pid}.
Moreover, if $p=2$, then $E$ is a subring of $A_T$ containing $A$.
In the remainder of this paper, we study the factorization properties
of $E$. }
\end{example}


\section{The ring $D= \bigcup\limits_{n\in\mathbb{N}_0}F[X^{\frac{1}{p^n}}, X^{-\frac{1}{p^n}}]$ for a prime number $p$}

Let $p$ be a prime number, $F$ a field, $X$ an indeterminate over $F$,
$D_n = F[X^{\frac{1}{p^n}}, X^{-\frac{1}{p^n}}]$ for each $n\in\mathbb{N}_0$
and $D = \bigcup\limits_{n\in\mathbb{N}_0}D_n$.
Then $D_n$ is a PID, $D_n \subsetneq D_{n+1}$ and $D_{n+1}$ is integral over $D_n$ for all $n \in \mathbb{N}_0$.
Hence, $D_0 \subsetneq D_1 \subsetneq \cdots \subsetneq D$
is a Dedekind tower-construction and $\{D_n\}_{n \in \mathbb{N}_0}$ is a system of Dedekind domains. For a nonzero polynomial $f\in F[X^{\frac{1}{p^n}}]$, we denote by $\textnormal{deg}_n(f)$ the degree of $f$ as an element of $F[X^{\frac{1}{p^n}}]$
and $f(0) \neq 0$ means that the constant term of $f$ is nonzero. For example, if $f(X) = X+1$,
then deg$_n(f(X)) = p^n$ for all $n \in \mathbb{N}_0$ and $f(0) \neq 0$.

In this section we analyze some factorization properties of
$D = \bigcup\limits_{n\in\mathbb{N}_0}D_n$ with a focus on the case of char$(F)=p$.
By Corollary \ref{thm31}, $D= \bigcup\limits_{n\in\mathbb{N}_0}D_n$ is a group ring, and
if char$(F)\neq p$, then $D$ is the ``simplest" monoid domain that is almost Dedekind.
However, as mentioned after Corollary \ref{thm31}, we remark that analyzing the group ring $D= \bigcup\limits_{n\in\mathbb{N}_0}D_n$
via a system of Dedekind domains is beneficial when one wishes to find prime elements and their factorization properties.
We first give the ring-theoretic property of
$D= \bigcup\limits_{n\in\mathbb{N}_0}D_n$ in the following proposition.

\begin{proposition} \label{prop3.1}
Let $p$ be a prime number, $F$ a field, $X$ an indeterminate over $F$,
$D_n=F[X^{\frac{1}{p^n}}, X^{-\frac{1}{p^n}}]$ for each $n\in\mathbb{N}_0$
and $D=\bigcup\limits_{n\in\mathbb{N}_0}D_n$.
Then the following statements are satisfied.
\begin{enumerate}[font=\normalfont]
\item $D$ is a B{\'e}zout domain of Krull dimension one.
\item $D$ is not a Dedekind domain.
\item If $M$ is a maximal ideal of $D$ such that $D/M$ is finite,
then $F$ is finite and $M$ is not a principal ideal of $D$.
\item dim$(D[[X]]) \geq 2^{\aleph_1}$, and equality holds under the continuum hypothesis if $F$ is countable.
\end{enumerate}
\end{proposition}

\begin{proof}
(1) It is clear that $\{D_n\}_{n \in \mathbb{N}_0}$ is a system of Dedekind domains
such that $D_n$ is a PID for all $n \in \mathbb{N}_0$.
Hence, $D$ is a B{\'e}zout domain of Krull dimension one by Corollary \ref{pid}(1).

(2)  Consider the prime ideal $P_n=(X^{\frac{1}{p^n}}-1)D_n$
and a proper ideal $Q_n=(\sum\limits_{i=0}^{p-1}X^{\frac{i}{p^n}})D_n$ of $D_n$ for each $n\in\mathbb{N}_0$.
Then $P_nD_{n+1}=P_{n+1}Q_{n+1}$ for each $n\in\mathbb{N}_0$, so $P_0\subseteq P_n$
and $P_n$ is never inert in $D_{n+1}$ for each $n\in\mathbb{N}_0$.
Hence, by Corollary \ref{dede}, $D$ cannot be a Dedekind domain.

(3) Note that $D/M=\bigcup\limits_{i\in\mathbb{N}_0}(D_i/(M\cap D_i))$ for each maximal ideal $M$ of $D$.
Note also that $M\cap D_i = f_iD_i$ for some irreducible element $f_i$ of $F[X^{\frac{1}{p^i}}]$ for all $i \in \mathbb{N}_0$.
Hence, $[D_i/(M\cap D_i):F]=\textnormal{deg}_{i}(f_i)$. Thus, if $D/M$ is finite,
then $F$ is finite. On the other hand, suppose that $M$ is principal. Then $M=f_iD$
for some $f_i\in F[X^{\frac{1}{p^i}}]$ with $i\in\mathbb{N}_0$. Note that $f_i$
is an irreducible element of $D_n$ for every $n\ge i$ by Lemma \ref{lemma1.1}(3).
Hence, $\textnormal{deg}_{n+1}(f_i)=p\cdot\textnormal{deg}_{n}(f_i)$ for each $n\ge i$,
and thus $D/M$ is infinite by Lemma \ref{lemma2.3}(6).

(4) It is easy to see that if $F$ is countable, then $D_n$ is countable for each $n \in \mathbb{N}_0$,
and hence $D$ is countable. Thus, the result follows directly from (2) above and Corollary \ref{coro2.7}(2).
\end{proof}

A nonzero nonunit $q$ of an integral domain $E$ is said to be \textit{primary} if
$qE$ is a primary ideal. A \textit{weakly factorial domain} is an integral domain whose
nonzero nonunits can be written as a finite product of primary elements \cite{az90}.
Every prime element is primary, so each UFD is a weakly factorial domain.
Hence, the class of weakly factorial domains includes UFDs, PIDs and one-dimensional quasilocal domains.

\begin{corollary} \label{coro-wfd}
Let $p$ be a prime number, $F$ a field of characteristic $p$, $X$ an indeterminate over $F$,
$D_n=F[X^{\frac{1}{p^n}}, X^{-\frac{1}{p^n}}]$ for each $n\in\mathbb{N}_0$
and $D=\bigcup\limits_{n\in\mathbb{N}_0}D_n$.
\begin{enumerate}[font=\normalfont]
\item Let $f(X) \in F[X]$ be a nonconstant polynomial with $f(0) \neq 0$.
Then $f(X)$ is a primary element of $D$ if and only if $f(X) = ug(X)^m$ for some irreducible polynomial $g(X) \in F[X]$, $m \in \mathbb{N}$
and $u \in F$.
\item $D$ is a weakly factorial domain.
\item Let $\{X_{\alpha}\}$ be a nonempty set of indeterminates over $D$. Then
$D[\{X_{\alpha}\}]$ is a weakly factorial domain.
\end{enumerate}
\end{corollary}

\begin{proof}
Note that every nonzero prime ideal of $D$ is maximal by Proposition \ref{prop3.1}(1),
so a nonzero nonunit $a \in D$ is primary if and only if $\sqrt{aD}$ is a prime ideal.

(1) $(\Rightarrow)$ If $f(X)$ is primary, then $M:= \sqrt{f(X)D}$ is a maximal ideal of $D$,
and hence $M \cap D_0 = g(X)D_0$ for some irreducible $g(X) \in F[X]$ by Lemma \ref{lemma1.1}(2).
Moreover, $g(X)D_0 = \sqrt{f(X)D} \cap D_0 = \sqrt{f(X)D_0}$ by Lemma \ref{lemma1.1}(4).
Thus, $f(X) = ug(X)^m$ for some $m \in \mathbb{N}$ and $u \in F$.
$(\Leftarrow$) Assume that there are at least two distinct maximal ideals of $D$, say, $M_1, M_2$,
such that $M_1 \cap D_0 = M_2 \cap D_0 = g(X)D_0$. Then $M_1 \cap D_n \neq M_2 \cap D_n$ for some $n \in \mathbb{N}_0$.
Let $h(X^{\frac{1}{p^n}}) \in M_1 \cap D_n$. Then char$(F)=p$ implies $h(X^{\frac{1}{p^n}})^{p^n} \in M_1 \cap D_0 \subseteq M_2 \cap D_n$,
and since $M_2 \cap D_n$ is a prime ideal, $h(X^{\frac{1}{p^n}}) \in M_2 \cap D_n$. Hence, $M_1 \cap D_n \subseteq M_2 \cap D_n$,
and thus $M_1 \cap D_n = M_2 \cap D_n$, a contradiction. Therefore,
$\sqrt{f(X)D} = \sqrt{g(X)D}$ is a maximal ideal.

(2) Let $a \in D$ be a nonzero nonunit. We may assume that $a \in F[X^{\frac{1}{p^n}}] \subsetneq D_n$,
with $a(0) \neq 0$, for some $n \in \mathbb{N}_0$.
Then $a$ is a finite product of prime elements in $F[X^{\frac{1}{p^n}}]$, and since each prime element of $F[X^{\frac{1}{p^n}}]$
with nonzero constant term
is a primary element of $D$ by (1) above, $a$ is a finite product of primary elements of $D$. Thus, $D$ is a weakly factorial domain.

(3) Recall that $D[\{X_{\alpha}\}]$ is a weakly factorial domain if and only if $D$ is a weakly
factorial GCD-domain \cite[Theorems 7 and 13]{az95} and a B{\'e}zout domain is a GCD-domain.
Thus, by (2) and Proposition \ref{prop3.1}(1), $D[\{X_{\alpha}\}]$ is weakly factorial.
\end{proof}

Let $D$ be the integral domain of Proposition \ref{prop3.1}. We are concerned with the factorization of elements of $D$
and the first question one may come up with is the following: which element of $D$ is a prime element?
Since $D$ is a B{\'e}zout domain by Proposition \ref{prop3.1}, this is equivalent to the question of deciding irreducible elements of $D$.
The answer to this question is our next result, and
in doing so, we need the following result concerning irreducible polynomials over a field.

\begin{lemma}
    \label{thc}
Let $F$ be a field and $X$ an indeterminate over $F$.
\begin{enumerate}[font=\normalfont]
\item $($Capelli's theorem \cite{b57}$)$ If $f(X)$ and $g(X)$ are elements of $F[X]$,
 then $f(g(X))$ is an irreducible element of $F[X]$ if and only if
$f(X)$ is an irreducible element of $F[X]$ and $g(X)-\alpha$ is an irreducible element of $F(\alpha)[X]$,
where $\alpha$ is a root of $f(X)$ (in an algebraic closure of $F$).
\item \cite[Theorem 8.1.6]{ka89} If $a\in F$ and $n\in\mathbb{N}$,
then $X^n-a$ is an irreducible polynomial of $F[X]$ if and only if $a\not\in F^p:=\{x^p\mid x\in F\}$ for each prime number $p$
that divides $n$ and $a\not\in -4F^4:=\{-4x^4\mid x\in F\}$ whenever $4\mid n$.
\end{enumerate}
\end{lemma}

Let $D$ be the integral domain of Proposition \ref{prop3.1}
and $f\in D$ a nonzero element. Then $f\in F[X^{\frac{1}{p^n}}, X^{-\frac{1}{p^n}}]$ for some $n\in\mathbb{N}_0$.
Hence, $f=(X^{\frac{1}{p^n}})^mg(X^{\frac{1}{p^n}})$ for some $m\in\mathbb{N}_0$ and
$g(X^{\frac{1}{p^n}})\in F[X^{\frac{1}{p^n}}]$ with $g(0)\neq 0$. It is clear that $(X^{\frac{1}{p^n}})^m$ is a unit of $D$,
so $f$ is irreducible in $D$ if and only if $g(X^{\frac{1}{p^n}})$ is irreducible in $\bigcup\limits_{n\in\mathbb{N}_0}F[X^{\frac{1}{p^n}}]$.
It is also clear that if $k \in \mathbb{N}_0$,
then $g(X^{\frac{1}{p^n}}) = g((X^{\frac{1}{p^{n+k}}})^{p^k})$ in $F[X^{\frac{1}{p^{n+k}}}]$,
so $g(X^{\frac{1}{p^n}})$ is irreducible in $F[X^{\frac{1}{p^{n+k}}}]$ if and only if $g(X^{p^k})$ is irreducible in $F[X]$.

\begin{proposition}
\label{coro3.2}
Let $p$ be a prime number, $F$ a field, $X$ an indeterminate over $F$,
$D_n=F[X^{\frac{1}{p^n}}, X^{-\frac{1}{p^n}}]$ for each $n\in\mathbb{N}_0$
and $D=\bigcup\limits_{n\in\mathbb{N}_0}D_n$.
Let $f(X^{\frac{1}{p^m}}) \in F[X^{\frac{1}{p^m}}]$ be an irreducible polynomial in $D_m$ with $f(0)\neq 0$
for some $m \in \mathbb{N}_0$.
Then the following statements are equivalent.
\begin{enumerate}[font=\normalfont]
\item $f(X^{\frac{1}{p^m}})$ is an irreducible element of $D$.
\item $f(X^{\frac{1}{p^m}})$ is an irreducible element of $E$, where
\[E =
\begin{cases} D_{m+2} &\textnormal{ if $p= 2$},\\
D_{m+1} &\textnormal{ if $p > 2$}.
\end{cases}
\]
\item $f(X^{\frac{1}{p^{m}}})$ is an irreducible element of $D_{m+k}$ for all $k \in \mathbb{N}_0$.
\end{enumerate}
In particular, if $F$ is algebraically closed, then $D$ has no irreducible element and each maximal ideal of $D$ is not invertible.
\end{proposition}

\begin{proof}
(1) $\Rightarrow$ (2) This is clear.

(2) $\Rightarrow$ (3) Let $\alpha$ be a root of $f(X)$ (in an algebraic closure of $F$) and
$k\in\mathbb{N}_0$. Then $f(X^{p^k})$ is irreducible in $F[X]$ if and only if $f(X)$ is irreducible in $F[X]$
and the polynomial $X^{p^k}-\alpha$ is irreducible in $F(\alpha)[X]$ by Lemma \ref{thc}(1).

Suppose $p=2$. Then Lemma \ref{thc}(2) shows that the following three conditions are equivalent:
\begin{enumerate}[label=(\roman*)]
    \item[$\cdot$] $X^4-\alpha$ is irreducible in $F(\alpha)[X]$.
    \item[$\cdot$] $X^{2^k}-\alpha$ is irreducible in $F(\alpha)[X]$ for each $k\ge 2$.
    \item[$\cdot$] $\alpha\not\in F(\alpha)^2$ and $\alpha\not\in -4F(\alpha)^4$.
\end{enumerate}
Hence, $f(X^{\frac{1}{2^{m}}})$ is irreducible in $D_{m+k}$ if and only if
$f(X^{2^k})$ is irreducible in $F[X]$, if and only if $f(X^4)$ is irreducible in $F[X]$,
if and only if $f(X^{\frac{1}{2^m}})$ is irreducible in $D_{m+2}$.
Next, if $p>2$, then $X^{p^k}-\alpha$ is irreducible in $F(\alpha)[X]$
if and only if $\alpha\not\in F(\alpha)^p$, if and only if $X^p - \alpha$ is irreducible in $F(\alpha)[X]$ by Lemma \ref{thc}(2).
Thus, $f(X^{\frac{1}{p^{m}}})$ is irreducible in $D_{m+k}$ if and only if
$f(X^{p^{k}})$ is irreducible in $F[X]$, if and only if
$f(X^{p})$ is irreducible in $F[X]$, if and only if $f(X^{\frac{1}{p^{m}}})$ is irreducible in $D_{m+1}$.

(3) $\Leftrightarrow$ (1) This follows directly from Lemma \ref{lemma1.1}(4).

In particular, if $F$ is algebraically closed, then $f(X^{\frac{1}{p^n}})$
is not irreducible in both $D_{n+1}$ and $D_{n+2}$ because deg$_{k}(f(X^{\frac{1}{p^n}})) \geq 2$
for every $k \in \mathbb{N}_0$ with $k \geq n+1$, and thus
$f(X^{\frac{1}{p^n}})$ is not irreducible in $D$. This also implies that every maximal ideal of $D$ is not principal.
Thus, as $D$ is a B{\'e}zout domain, every maximal ideal of $D$ is not invertible.
\end{proof}

\begin{corollary}
Let $f(X) \in F[X]$ be an irreducible polynomial with $f(0) \neq 0$
and $p$ an odd prime number. Then $f(X^p)$ $($resp., $f(X^4))$ is irreducible in $F[X]$
if and only if $f(X^{p^n})$ $($resp., $f(X^{2^n}))$ is irreducible in $F[X]$ for all $n \in \mathbb{N}$.
\end{corollary}

\begin{proof}
This follows from the equivalence of (2) and (3) in Proposition \ref{coro3.2} by observing that
$f(X) = f((X^{\frac{1}{q^{n}}})^{q^n})$, so $f(X)$ is irreducible in $D_{n}$
if and only if $f(X^{q^n})$ is irreducible in $F[X]$ for a prime number $q$.
\end{proof}

For $d\in\mathbb{N}$,
let $\Phi_d(X) \in \mathbb{Z}[X]$ be the $d$-th cyclotomic polynomial. It is useful to note that if
$D = \bigcup\limits_{n\in\mathbb{N}_0}\mathbb{Q}[X^{\frac{1}{p^n}}, X^{-\frac{1}{p^n}}]$ for a prime number $p$,
then $\Phi_d(X)$ is divided by (countably) infinitely many prime elements in $D$
if and only if $p \nmid d$, as the following corollary shows.

\begin{corollary} \label{excyclo}
Let $p$ be a prime number, $X$ an indeterminate over $\mathbb{Q}$,
$D_n$ = $\mathbb{Q}[X^{\frac{1}{p^n}}, X^{-\frac{1}{p^n}}]$ for each $n\in\mathbb{N}_0$
and $D=\bigcup\limits_{n\in\mathbb{N}_0}D_n$. Then for each $d\in\mathbb{N}$, the following statements hold.
\begin{enumerate}[font=\normalfont]
\item  $p\mid d$ if and only if $\Phi_d(X)$ is a prime element of $D$.
\item If $p\nmid d$, then $\Phi_d(X)$ is contained in countably infinitely many maximal ideals of $D$.
\end{enumerate}
\end{corollary}

\begin{proof}
It is well-known \cite[p.280]{l02} that
\[\Phi_{d}(X^p)=
\begin{cases}\Phi_{pd}(X)\Phi_d(X) &\textnormal{ if $p\nmid d$},\\
\Phi_{pd}(X) &\textnormal{ if $p\mid d$}.
\end{cases}
\]
Hence, if $p\mid d\in\mathbb{N}$, we have $\Phi_d(X^{p^2})=\Phi_{pd}(X^p)=\Phi_{p^2d}(X)$,
which is irreducible over $\mathbb{Q}$, so $\Phi_d(X)$ is a prime element of $D$ by Propositions \ref{coro3.2} and \ref{prop3.1}(1).
It also follows that if $p\nmid d$, then $\Phi_d(X^{p^n})=\Phi_d(X)\prod\limits_{i=1}^{n}\Phi_{p^id}(X)$ for each $n\in\mathbb{N}$.
Therefore, by Lemma \ref{lemma1.1}(1), the set of maximal ideals of $D$ containing
$\Phi_d(X)$ is $\{N\}\cup\{\Phi_{pd}(X^{\frac{1}{p^i}})D\mid i\in\mathbb{N}\}$,
where $N$ is the ideal of $D$ generated by $\{\Phi_d(X^{\frac{1}{p^i}})\mid i\in\mathbb{N}\}$,
because dim$(D)=1$ by Proposition \ref{prop3.1}.
In particular, if $\Phi_d(X)$ is a prime element of $D$, then
$p\mid d$ by (2).
\end{proof}

Let $D = \bigcup\limits_{n\in\mathbb{N}_0}\mathbb{Q}[X^{\frac{1}{p^n}}, X^{-\frac{1}{p^n}}]$ for a prime number $p$.
Then, by Corollary \ref{excyclo}, we can find a sufficiently many prime elements of $D$.
We next give another way of how to construct a prime element of $D$.

\begin{example} \label{ex1}
{\em Let $F=\mathbb{Q}$, $p$ a prime number and $f(X)=a_0+a_1X+\cdots+a_mX^m$ a nonconstant polynomial of $\mathbb{Z}[X]$
such that $q\mid a_0, q\mid a_1,\dots, q\mid a_{m-1}, q\nmid a_m, q^2\nmid a_0$ for some prime number $q$.

(1) $f(X)$ is irreducible in $\mathbb{Q}[X^{\frac{1}{p^n}}]$ for all $n\in\mathbb{N}_0$
by Eisenstein's criterion, and hence
$f(X)$ is irreducible in $D$ by Lemma \ref{lemma1.1}(3). Thus, $f(X)D$ is a principal prime ideal of $D$.

(2) Let $M_n = f(X)D \cap D_n$ for some $n \in \mathbb{N}_0$.
Then $M_n = f(X)D_n$ and deg$_nf(X) = m\cdot p^n$.
Thus, $[D_n/M_n:\mathbb{Q}] = m\cdot p^n$ and $[D_{n+1}/M_{n+1}:D_n/M_n]=p$.}
\end{example}

By Proposition \ref{coro3.2},
we are interested in how an element of $F[X^{\frac{1}{p^n}},X^{-\frac{1}{p^n}}]$
factors in $F[X^{\frac{1}{p^{n+1}}},X^{-\frac{1}{p^{n+1}}}]$ for $n \in \mathbb{N}_0$ and a prime number $p$.
The case when $\textnormal{char}(F)=p$ gives a neat answer to this question.

\begin{lemma}
\label{lemma1.5}
Let $p$ be a prime number, $F$ a field of characteristic $p$ and $X$ an indeterminate over $F$.
If $f(X)\in F[X]$ is irreducible as an element of $F[X]$ but not as an element of $F[X^{\frac{1}{p}}]$,
then there exists an irreducible element $g(X^{\frac{1}{p}})$ of $ F[X^{\frac{1}{p}}]$
such that $f(X)F[X]=g(X^{\frac{1}{p}})^{p}F[X]=g(X^{\frac{1}{p}})F[X^{\frac{1}{p}}]\cap F[X].$
\end{lemma}

\begin{proof}
Since $f(X)$ is not irreducible in $F[X^{\frac{1}{p}}]$, we have $f(X)=g_1(X^{\frac{1}{p}})\cdots g_k(X^{\frac{1}{p}})$
for some (not necessarily distinct) irreducible elements $\{g_i(X^{\frac{1}{p}})\}_{i=1}^{k}$ of $F[X^{\frac{1}{p}}]$.
Note that $g_1(X^{\frac{1}{p}})F[X^{\frac{1}{p}}]\cap F[X]$ contains $f(X)$. Since $g_1(X^{\frac{1}{p}})F[X^{\frac{1}{p}}]\cap F[X]$
is a prime ideal of $F[X]$, we must have $f(X)F[X]=g_1(X^{\frac{1}{p}})F[X^{\frac{1}{p}}]\cap F[X]$.
Note also that $g_1(X^{\frac{1}{p}})^{p}\in g_1(X^{\frac{1}{p}})F[X^{\frac{1}{p}}]\cap F[X]$, so $g_1(X^{\frac{1}{p}})^{p}=f(X)h(X)$
for some $h(X)\in F[X]$. Now, $g_1(X^{\frac{1}{p}})^{p}=g_1(X^{\frac{1}{p}})\cdots g_k(X^{\frac{1}{p}}) h(X)$.
Since $F[X^{\frac{1}{p}}]$ is a UFD whose units are nonzero elements of $F$, it follows that there exist nonzero $u_i\in F$
such that $u_ig_i(X^{\frac{1}{p}})=g_1(X^{\frac{1}{p}})$ for each $i$ and $h(X)=u_1\cdots u_kg_1(X^{\frac{1}{p}})^{p-k}$.
Since $k>0$ and $h(X)\in F[X]$, we must have $p=k$. Thus, $f(X)=ug_1(X^{\frac{1}{p}})^{p}$, where $u=u_1^{-1}\cdots u_k^{-1}$,
from which the conclusion follows.
\end{proof}

Let $X^1(E)$ be the set of height one prime ideals of an integral domain $E$.
Then $E$ is called a \textit{generalized Krull domain} if the following three properties are satisfied: (i)
$E = \bigcap\limits_{P \in X^1(D)}E_P$, (ii) $E_P$ is a valuation domain for all $P \in X^1(E)$
and (iii) each nonzero nonunit of $E$ is contained in only finitely many prime ideals in $X^1(E)$.
Hence, almost Dedekind domain is a Dedekind domain if and only if it is a generalized Krull domain.
 A \textit{Krull domain} $E$ is a generalized Krull domain such that $E_P$ is a DVR for all $P \in X^1(E)$.
Mori-Nagata theorem says that the integral closure of a Noetherian domain is a Krull domain \cite[Theorem 33.10]{n62}.
We next show that $D = \bigcup\limits_{n\in\mathbb{N}_0}\mathbb{Q}[X^{\frac{1}{p^n}}, X^{-\frac{1}{p^n}}]$
is a generalized Krull domain if $\textnormal{char}(F)=p$.

\begin{theorem} \label{th1.5}
Let $p$ be a prime number, $F$ a field of characteristic $p$, $X$ an indeterminate over $F$,
$D_n=F[X^{\frac{1}{p^n}}, X^{-\frac{1}{p^n}}]$ for each $n\in\mathbb{N}_0$
and $D=\bigcup\limits_{n\in\mathbb{N}_0}D_n$. Then the following statements are satisfied.
\begin{enumerate}[font=\normalfont]
\item $D$ is not an almost Dedekind domain.
\item $D$ is of finite character.
\item $D$ is a one-dimensional generalized Krull domain.
\item The map $\varphi :$ Spec$(D) \rightarrow$ Spec$(D_0)$ given by $Q \mapsto Q \cap D_0$ is a one-to-one correspondence.
\item Let $M = \bigcup\limits_{n \in \mathbb{N}_0}(X^{\frac{1}{p^n}}-1)D_n$.
Then $M$ is a maximal ideal of $D$, $D_M$ is a rank one nondiscrete valuation domain and $D/M \cong F$.
\end{enumerate}
\end{theorem}

\begin{proof}
(1) This follows directly from Corollary \ref{thm31}.

(2) Let $f_0(X)$ be an irreducible element of $F[X]$. For each $k \in\mathbb{N}_0$,
there exists an irreducible polynomial $f_k(X^{\frac{1}{p^k}})\in F[X^{\frac{1}{p^k}}]$
such that either $f_k(X^{\frac{1}{p^k}}) = f_{k+1}(X^{\frac{1}{p^{k+1}}})$ or
$f_k(X^{\frac{1}{p^k}}) = f_{k+1}(X^{\frac{1}{p^{k+1}}})^p$ by Lemma \ref{lemma1.5} and induction on $k$.
Let $N$ be the ideal of $D$ generated by $\{f_k(X^{\frac{1}{p^k}})\mid k\in\mathbb{N}_0\}$.
Then $N\cap D_k$ is generated by $f_k(X^{\frac{1}{p^k}})$ for each $k\in\mathbb{N}_0$,
so $N$ is a maximal ideal of $D$ by Theorem \ref{th1}(1).
It is easy to see that $N$ is the only maximal ideal of $D$ that contracts to $f_0(X)D_0$.
Hence, for each $k\in\mathbb{N}_0$, an irreducible element of $D_k$ is contained in a unique maximal ideal of $D$.
Therefore, each nonunit of $D$ is contained in only finitely many maximal ideals of $D$,
because $D_k$ is a PID for all $k \in \mathbb{N}_0$.

(3) By Proposition \ref{prop3.1}, $D$ is a one-dimensional B{\'e}zout domain.
Thus, by (2), $D$ is a generalized Krull domain.

(4) Let $Q_1$ and $Q_2$ be two distinct prime ideals of $D$.
Then there is an $n \in \mathbb{N}_0$ such that $Q_1 \cap D_n \nsubseteq Q_2 \cap D_n$, so
we can choose $f(X^{\frac{1}{p^n}}) \in Q_1 \cap D_n \setminus Q_2 \cap D_n$.
Hence, $f(X^{\frac{1}{p^n}})^{p^n}  \in Q_1 \cap D_0 \setminus Q_2 \cap D_0$,
and thus, $Q_1 \cap D_0 \neq Q_2 \cap D_0$. Therefore, $\varphi$ is injective. Moreover, if
$P$ is a prime ideal of $D_0$, then $PD \neq D$, and hence $PD \subseteq Q$ for some prime ideal $Q$ of $D$.
Thus, $Q \cap D_0 = P$. Therefore, $\varphi$ is surjective.

(5) It is clear that $(X^{\frac{1}{p^n}}-1)D_n$ is a prime ideal of $D_n$ such that
$$(X^{\frac{1}{p^n}}-1)D_n \subsetneq (X^{\frac{1}{p^{n+1}}}-1)D_{n+1}$$ for all $n \in \mathbb{N}_0$.
Hence, $M$ is a maximal ideal of $D$ by Theorem \ref{th1}(1) and ht$M=1$ by Proposition \ref{coro3.2}(1).
It is clear that $MD_M$ is not finitely generated. Thus, $D_M$ is a rank-one nondiscrete valuation domain,
because $D$ is a B{\'e}zout domain.

Now note that $M \cap D_n = (X^{\frac{1}{p^n}}-1)D_n$;
so $D_n/(M \cap D_n) = F$. Hence, if $f \in D$, then $f \in D_n$ for some $n \in \mathbb{N}_0$,
and thus $f+ (M \cap D_n) = a + (M \cap D_n)$ for some $a \in F$. Therefore, $D/M = F$.
\end{proof}

\begin{remark}
    {\em
Let $E[[X]]$ be the power series ring over an integral domain $E$
and $c(f)$ the ideal of $E$ generated by the coefficients of $f \in E[[X]]$. Then $E[[X]]$
is a generalized Krull domain if and only if $E$ is a Krull domain \cite[Theorem 1]{gkt20},
so if $E$ is a generalized Krull domain but not a Krull domain, then $E[[X]]$ is not a
generalized Krull domain. However, if $E$ is a generalized Krull domain
and $N_v = \{f \in E[[X]] \mid f \neq 0$ and $c(f)^{-1} = E\}$,
then $E[[X]]_{N_v}$ is a one-dimensional generalized Krull domain \cite[Corollary 3.9]{c20}.
Thus, if $E$ is the ring of Theorem \ref{th1.5}, $E[[X]]_{N_v}$ is a one-dimensional generalized Krull domain.}
\end{remark}




Now, we establish a sufficient condition for $D$ to have no irreducible elements.

\begin{corollary} \label{coro of th1.5}
Let $p$ be a prime number, $F$ a field of characteristic $p$, $X$ an indeterminate over $F$,
$D_n=F[X^{\frac{1}{p^n}}, X^{-\frac{1}{p^n}}]$ for each $n\in\mathbb{N}_0$
and $D=\bigcup\limits_{n\in\mathbb{N}_0}D_n$. For $n \in \mathbb{N}_0$, let
$f(X^{\frac{1}{p^{n+1}}}) \in F[X^{\frac{1}{p^{n+1}}}]$ be a nonconstant polynomial
with $f(0) \neq 0$. Then the following statements are satisfied.
\begin{enumerate}[font=\normalfont]
\item $f(X^{\frac{1}{p^{n+1}}})^p \in D_n$.
\item If $f(X^{\frac{1}{p^{n+1}}})\not\in D_n$, then $f(X^{\frac{1}{p^{n+1}}})$ is irreducible in $D_{n+1}$
if and only if $f(X^{\frac{1}{p^{n+1}}})^p$ is irreducible in $D_{n}$.
\item Let $F$ be an algebraic extension of $\mathbb{Z}/p\mathbb{Z}$.
Then each nonzero nonunit of $D_n$ is not irreducible in $D_{n+1}$.
In particular, $D$ has no irreducible element.
\item $F$ is finite if and only if every residue field of $D$ is finite.
\end{enumerate}
\end{corollary}

\begin{proof}
Let $f = f(X^{\frac{1}{p^{n+1}}})$ for convenience.

(1) $f^p \in D_n$ because char$(F)=p$.

(2) Assume that $f$ is irreducible in $D_{n+1}$ and let $Q = fD_{n+1} \cap D_n$. Then $Q$ is a prime ideal of $D_n$,
and since $D_n$ is a PID, $Q = gD_n$ for some $g \in D_n$. Note that $f^p \in Q$
and $\sqrt{gD_{n+1}} = fD_{n+1}$ by Theorem \ref{th1.5}(4). Hence,
$f^p = ug$ for some $0 \neq u \in F$ by a simple calculation. Thus, $f^p$
is irreducible in $D_n$. Conversely, assume that $f^p$ is irreducible in $D_n$. Then, by Theorem \ref{th1.5}(4) again,
$\sqrt{fD_{n+1}}$ is a prime ideal, because $f^p$
is irreducible in $D_n$. Now, let $h \in D_{n+1}$ be such that $hD_{n+1} =  \sqrt{fD_{n+1}}$.
Then, as $D_{n+1}$ is a PID, $f= uh^k$ for some $u \in F$ and $k \in \mathbb{N}$, and hence
$f^p = u^p(h^p)^k$. Note that $h^p \in D_n$ and $f^p$ is irreducible in $D_n$. Thus, $k=1$, and hence
$f = uh$ is irreducible in $D_{n+1}$.

(3) Let $g \in D_n$ be a nonconstant polynomial. If $g$ is not irreducible in $D_n$, then $g$ is also not
irreducible in $D_{n+1}$, so we may assume that $g$ is irreducible in $D_n$.
Then, by Lemma \ref{lemma1.1}(1), there is a prime ideal $P$ of $D_{n+1}$ which lies over $gD_n$.
For convenience, let $P = fD_{n+1}$ and $F_0 = (\mathbb{Z}/p\mathbb{Z})(c(f))$. Then $F_0$ is a finite field with $F_0 \subseteq F$.
Therefore, $a^{|F_0|} = a$ for all $a \in F_0$, so if $f \in D_n$, there
is an $h \in D_{n+1}$ and $0 \neq u \in F$ such that $f = uh^p$, a contradiction.
Hence, by the proof of (2) above, $g = uf^p$ for some $0 \neq u \in F$.
Thus, $g$ is not irreducible in $D_{n+1}$.

In particular, if $b \in D$ is nonconstant, then $b \in D_m$ for some $m \in \mathbb{N}_0$.
Hence, $b$ is not irreducible in $D_{m+1}$ by the previous paragraph, and thus $b$ is not irreducible in $D$.

(4) Suppose that $F$ is a finite field, and let $M$ be a maximal ideal of $D$. Then $M$ is not a principal ideal of $D$ by (3).
Hence, $M\cap D_n = f_nD_n$, where $f_n \in F[X^{\frac{1}{p^n}}]$ is an irreducible element of $D_n$
that is not irreducible in $D$, for all $n \in \mathbb{N}_0$ by Lemma \ref{lemma1.1}(3).
Then, by (3) above and Lemma \ref{lemma1.5}, deg$_n(f_n)=$ deg$_0(f_0)$ for all $n \in \mathbb{N}_0$. It then follows that
$|D_n/(M\cap D_n)| = |D_0/(M\cap D_0)|$ for every $n \in \mathbb{N}_0$. Thus, $D/M$ is a finite field by Lemma \ref{lemma2.3}(6).
The converse follows because $F$ is contained in every residue field of $D$.
\end{proof}

In the next corollary, we construct a one-dimensional Pr{\"u}fer domain
with infinitely many maximal ideals
(none of them are locally finitely generated) and has no irreducible elements.

\begin{corollary} \label{coro3.5}
Let $p$ be a prime number, $F$ an algebraically closed field of characteristic $p$, $X$ an indeterminate over $F$,
$D_n=F[X^{\frac{1}{p^n}}, X^{-\frac{1}{p^n}}]$ for each $n\in\mathbb{N}_0$
and $D=\bigcup\limits_{n\in\mathbb{N}_0}D_n$.
Let $P_{\alpha} = \bigcup\limits_{n \in \mathbb{N}_0} (X^{\frac{1}{p^n}}-\alpha^{\frac{1}{p^n}})D_n$
for each nonzero $\alpha \in F$. Then the following statements hold.
\begin{enumerate}[font=\normalfont]
\item Max$(D) = \{P_{\alpha} \mid 0 \neq \alpha \in F\}$,
\item $D_{P_{\alpha}}$ is a rank one nondiscrete valuation domain,
\item $D/P_{\alpha} \cong F$ and
\item ht$(P_{\alpha}[[X]]/P_{\alpha}D[[X]]) \geq 2^{\aleph_1}$ for each nonzero $\alpha \in F$.
\end{enumerate}
\end{corollary}

\begin{proof}
(1) This follows from Theorems \ref{th1.5}(4), \ref{th1}(1) and the fact that $\alpha^{\frac{1}{p^{n+1}}} \in F$ and
$X^{\frac{1}{p^n}}-\alpha^{\frac{1}{p^n}} = (X^{\frac{1}{p^{n+1}}}-\alpha^{\frac{1}{p^{n+1}}})^p$
for all $n \in \mathbb{N}_0$.

(2) and (3) can be proved by an argument similar to the proof of Theorem \ref{th1.5}(5).

(4) It suffices to show that $P_{\alpha}$ is not an SFT-ideal by Corollary \ref{coro2.4} and \cite[Theorem 15]{tk21}.
Assume to the contrary that $P_{\alpha}$ is an SFT-ideal.
Then there is a finitely generated ideal $J$ of $D$ and  $k \in \mathbb{N}$
so that $a^k \in J$ for all $a \in P_{\alpha}$.
Hence, there exists $n \in \mathbb{N}$ with $J \subseteq (X^{\frac{1}{p^n}}-\alpha^{\frac{1}{p^n}})D$,
so $(X^{\frac{1}{p^m}}-\alpha^{\frac{1}{p^m}})^{p^k} \in (X^{\frac{1}{p^n}}-\alpha^{\frac{1}{p^n}})D$ for all $m \in \mathbb{N}$.
In particular, $X^{\frac{1}{p^{n+1}}}-\alpha^{\frac{1}{p^{n+1}}} \in (X^{\frac{1}{p^n}}-\alpha^{\frac{1}{p^n}})D$, which implies
that $X^{\frac{1}{p^{n+1}}}-\alpha^{\frac{1}{p^{n+1}}}$ is a unit of $D$, a contradiction.
\end{proof}


\section{The ring $D= \bigcup\limits_{n\in\mathbb{N}_0}F[X^{\frac{1}{2^n}}, X^{-\frac{1}{2^n}}]$}

In this section, we continue our investigation on $D= \bigcup\limits_{n\in\mathbb{N}_0}F[X^{\frac{1}{p^n}}, X^{-\frac{1}{p^n}}]$
with a focus on the specific case $p=2$, because
the key result of this section follows from the simple fact that if $p=2$,
then $F$ has a primitive $p$-th root of unity if and only if $\textnormal{char}(F)\neq p$.
(For the proof, note that if $a\in F$ is a primitive 2nd root of unity, then $a^2=1$ but $a\neq 1$,
so $a=-1$ and $\textnormal{char}(F)\neq 2$. Conversely, if $\textnormal{char}(F)\neq 2$,
then $-1\neq 1$ and $-1$ is a primitive 2nd root of unity.)
Using this, we gain a clearer picture of factorization of polynomials in $D$, starting with Proposition \ref{prop4.3}.
Throughout the remainder of this paper,
we will assume the following notation.

\begin{notation} \label{notation}
{\em Let $F$ be a field, $X$ an indeterminate over $F$ and for each $n\in\mathbb{N}_0$
$$R_n =F[X^{\frac{1}{2^n}}], \ \ D_n =F[X^{\frac{1}{2^n}}, X^{-\frac{1}{2^n}}], \ \
R =\bigcup\limits_{n\in\mathbb{N}_0}R_n, \ \ D=\bigcup\limits_{n\in\mathbb{N}_0}D_n.$$
Then $\{R_n\}_{n \in \mathbb{N}_0}$ and $\{D_n\}_{n \in \mathbb{N}_0}$ are systems of Dedekind domains such that both $R_n$ and $D_n$
are PIDs for all $n \in \mathbb{N}_0$.
Hence, by Corollary \ref{pid}(1), $R$ and $D$ are B{\'e}zout domains of Krull dimension one
and $D=R_S$ for $S = \{X^{\frac{1}{2^n}} \mid n \in \mathbb{N}_0\}$. }
\end{notation}

It is clear that for each $n\in
\mathbb{N}_0$, an element of $R_n$ with nonzero constant term is irreducible in $R_n$ if and only if it is irreducible in $D_n$. From now on, we will freely use this fact
without further comments.

\begin{proposition} \label{prop4.3}
Suppose that $\textnormal{char}(F)\neq 2$ and $f(X^{\frac{1}{2^n}})$ is an irreducible polynomial of $R_n$ with $f(0)\neq 0$ for some $n\in\mathbb{N}_0$.
\begin{enumerate}[font=\normalfont]
\item If $f(X^{\frac{1}{2^n}})$ is not an element of $D_{n-1}$, then $f(X^{\frac{1}{2^n}})f(-X^{\frac{1}{2^n}})$ is an irreducible element of $D_{n-1}$.
\item If $f(X^{\frac{1}{2^n}})$ is not an irreducible element of $D_{n+1}$,
then $$f(X^{\frac{1}{2^n}})=ug(X^{\frac{1}{2^{n+1}}})g(-X^{\frac{1}{2^{n+1}}})$$ for some irreducible element $g(X^{\frac{1}{2^{n+1}}})$ of $D_{n+1}$
and $0 \neq u\in F$. In this case, $g(X^{\frac{1}{2^{n+1}}})$ and $g(-X^{\frac{1}{2^{n+1}}})$ are not associated as elements of $D_{n+1}$.
\item If $g(X) \in F[X]$ is an irreducible element of $D_0$, then $f(X^{\frac{1}{2^n}})^2$ does not divide $g(X)$ in $D$.
\item $\sqrt{f(X^{\frac{1}{2^n}})D}=f(X^{\frac{1}{2^n}})D$.
\end{enumerate}
\end{proposition}

\begin{proof}
(1) Let $h(X)=f(X^{\frac{1}{2^n}})f(-X^{\frac{1}{2^n}})$. Then $h(X)\in
R_{n-1}$ and $f(X^{\frac{1}{2^n}})\neq f(-X^{\frac{1}{2^n}})$ because $f(X^{\frac{1}{2^n}})\not\in D_{n-1}$.
Write $h(X)=ug_1(X^{\frac{1}{2^{n-1}}})\cdots g_m(X^{\frac{1}{2^{n-1}}})$
for some (not necessarily distinct) irreducible elements $g_1(X^{\frac{1}{2^{n-1}}}),\dots, g_m(X^{\frac{1}{2^{n-1}}})$ of
$R_{n-1}$ and  $u\in F$. Note that $g_1(X^{\frac{1}{2^{n-1}}}),\dots, g_m(X^{\frac{1}{2^{n-1}}})$
are even polynomials of $R_{n}$, while $f(X^{\frac{1}{2^n}})$ and $f(-X^{\frac{1}{2^n}})$ are not even.
So neither $f(X^{\frac{1}{2^n}})$ nor $f(-X^{\frac{1}{2^n}})$ is associated to any of
$g_1(X^{\frac{1}{2^{n-1}}}),\dots, g_m(X^{\frac{1}{2^{n-1}}})$ in $R_n$.
Since $g_1(X^{\frac{1}{2^{n-1}}}),\dots, g_m(X^{\frac{1}{2^{n-1}}})$ survives in $R_n$,
we conclude that $m=1$. Hence, $h(X)$ is an irreducible element of $R_{n-1}$,
and the statement follows.

(2) Let $f(X^{\frac{1}{2^n}})=\prod\limits_{i=1}^{m}h_i(X^{\frac{1}{2^{n+1}}})$ be a prime factorization of $f(X^{\frac{1}{2^n}})$ in $R_{n+1}$,
$$\mathcal{A} =\{i\in\{1,\dots, m\}\mid h_i(X^{\frac{1}{2^{n+1}}})\in R_{n}\} \text{ and }
\mathcal{B} =\{1,\dots, m\}\setminus \mathcal{A}.$$
Then since $f(X^{\frac{1}{2^n}})$ is an even function of $R_{n+1}$,
it follows that $\prod\limits_{i=1}^{m}h_i(-X^{\frac{1}{2^{n+1}}})$ is a prime factorization of $f(X^{\frac{1}{2^n}})$ in $R_{n+1}$.
Since $R_{n+1}$ is a UFD, for each $i\in \mathcal{B}$, there exists $j\in\mathcal{B}\setminus\{i\}$
so that $h_i(X^{\frac{1}{2^{n+1}}})$ is associated to $h_j(-X^{\frac{1}{2^{n+1}}})$ in $R_{n+1}$.
Hence, $\mathcal{B}$ has an even number of elements
and there exists a subset $\mathcal{C}\subseteq \mathcal{B}$ such that $|\mathcal{C}|=\cfrac{|\mathcal{B}|}{2}$ and $$f(X^{\frac{1}{2^n}})=u\big(\prod\limits_{i\in\mathcal{A}}h_i(X^{\frac{1}{2^{n+1}}})\big)\big(\prod\limits_{i\in\mathcal{C}}h_i(X^{\frac{1}{2^{n+1}}})h_i(-X^{\frac{1}{2^{n+1}}})\big)$$
for some $u\in F$.
Note that $h_i(X^{\frac{1}{2^{n+1}}})\in R_n$ for each $i\in\mathcal{A}$,
$h_i(X^{\frac{1}{2^{n+1}}})h_i(-X^{\frac{1}{2^{n+1}}})\in R_n$ for each $i\in\mathcal{C}$, and
$f(X^{\frac{1}{2^n}})$ is irreducible in $R_n$. Hence, $|\mathcal{A}\cup\mathcal{C}|=1$.
Moreover, since $f(X^{\frac{1}{2^n}})$ is not irreducible in $R_{n+1}$ by our assumption, we must have $\mathcal{A}=\emptyset$ and $|\mathcal{C}|=1$
by which we obtain the first assertion. The second assertion follows
from the fact that $h_i(X^{\frac{1}{2^{n+1}}})$ and
$h_i(-X^{\frac{1}{2^{n+1}}})$ are not associated for each $i\in\mathcal{C}$.

(3) Assume to the contrary that $f(X^{\frac{1}{2^n}})^2\mid g(X)$ in $D$.
Then $f(X^{\frac{1}{2^n}})^2\mid g(X)$ in $R_m$ for some $m \in \mathbb{N}_0$.
 Choose the smallest $k$ such that $f(X^{\frac{1}{2^n}})\in R_k$,
 and write $f(X^{\frac{1}{2^n}})=f_k(X^{\frac{1}{2^k}})$.
Then, by Corollary \ref{pid}(3), $f_k(X^{\frac{1}{2^k}})^2\mid g(X)$ in $R_k$.
 In other words, $f_k(X^{\frac{1}{2^k}})^2h_k(X^{\frac{1}{2^k}})=g_k(X^{\frac{1}{2^k}})=g(X)$ for some $h_k(X^{\frac{1}{2^n}})\in R_k$
 and $g_k(X^{\frac{1}{2^k}})\in R_k$.
 On the other hand, $k \ge 1$ since $g(X)$ is irreducible in $D_0$, so $f_{k-1}(X^{\frac{1}{2^{k-1}}}):=f(X^{\frac{1}{2^k}})f(-X^{\frac{1}{2^k}})$
 is an irreducible element of $D_{k-1}$ by (1) above.

 Since $f_k(X^{\frac{1}{2^k}})$ is an irreducible element of $D_k$, so is $f_k(-X^{\frac{1}{2^k}})$.
 Then $g_k(X^{\frac{1}{2^k}})=g_k(-X^{\frac{1}{2^k}})$ implies that $f_k(-X^{\frac{1}{2^k}})^2\mid g_k(X^{\frac{1}{2^k}})$ in $R_k$.
 Thus, $$f_{k-1}(X^{\frac{1}{2^{k-1}}})^2\mid g(X)  \text{ in }  R_k.$$
 Notice that $f_{k-1}(X^{\frac{1}{2^{k-1}}})^2\mid g(X)$ in $R_{k-1}$ by Corollary \ref{pid}(3).
Therefore, by induction on $k$, we deduce that $f_0(X)^2$ divides $g(X)$ in $R_0$
for some irreducible element $f_0(X)$ of $R_1$, a contradiction. Therefore,
$f(X^{\frac{1}{2^n}})^2 \nmid g(X)$ in $D$.

(4) Let $f =f(X^{\frac{1}{2^n}})$
and choose an element $a\in D$ such that $a\in \sqrt{fD}$.
It follows that  $a, f\in D_l$ and $a^m\in fD_l$ for some  $l,m\in\mathbb{N}$ with $l\ge n$.
Now, assume that $a = g_1^{e_1} \cdots g_k^{e_k}$ is a prime factorization of $a$
in $D_l$. Note that $g_i^2\nmid f$ in $D_l$ for $i=1, \dots, k$ by (3) above, because $l \geq n$ and $f$ is irreducible in $D_n$.
Hence, $a^m \in fD_l$ implies $f = ug_{i_1} \cdots g_{i_j}$ for some unit $u$ of $D_l$ and $i_1, \dots , i_j \in \{1, \dots , k\}$,
so $a \in fD_l\subseteq fD$. Thus, $fD$ is a radical ideal of $D$.
\end{proof}

An integral domain $E$ is called an \textit{SP domain}
if each proper ideal of $E$ can be written as a finite product of radical ideals of $E$ (SP is short for \textit{semiprime} (ideals),
which is a synonym of radical (ideals)).
The SP domains were first introduced by Vaughan and Yeagy \cite{vy78} who called them `domains with SP property'.
It is known that an SP domain is an almost Dedekind domain \cite[Theorem 2.4]{vy78} but not vice versa
(see, for example, \cite[Theorem 3.2]{y79} and \cite[p.122]{by76}). However, the next result shows that
the converse is true for the integral domain $D$ in Notation \ref{notation}.

\begin{theorem} \label{th2}
The following statements are equivalent.
\begin{enumerate}[font=\normalfont]
        \item $\textnormal{char}(F)\neq 2$.
        \item $D$ is an almost Dedekind domain.
        \item $D$ is an SP domain.
        \item Each nonunit of $D$ that is not a product of finitely many prime elements of $D$ is contained in infinitely many maximal ideals of $D$.
\item $D$ is not of finite character.
\item Each proper principal ideal of $D$ is a finite product of radical principal ideals of $D$.
\end{enumerate}
\end{theorem}

\begin{proof}
(1) $\Leftrightarrow$ (2) Corollary \ref{thm31}.

(2) $\Leftrightarrow$ (3) A maximal ideal $N$ of an integral domain $E$ is \textit{critical}
if each finitely generated subideal of $N$ is contained in $M^2$ for some maximal ideal $M$ of $E$.
Then $E$ is an SP domain if and only if $E$ is an almost Dedekind domain with no critical maximal ideal \cite[Theorem 2.1]{Ob05}.

Note that $D$ is a B{\'e}zout domain by Proposition \ref{prop3.1}(1),
so $D$ is an SP domain if and only if,
for each maximal ideal $N$ of $D$, there exists $a\in N$ such that $a\not\in M^2$ for every maximal ideal $M$ of $D$.
Now let $N$ be a maximal ideal of $D$, and choose an element $a$ of $D_0$ such that $N \cap D_0 = aD_0$.
Note that $a\in M^2$ for some maximal ideal $M$ of $D$ if and only if $m_k^2\mid a$ in $D_k$ for some $k\in\mathbb{N}_0$,
where $m_k$ is an element of $D_k$ such that $m_kD_k=M\cap D_k$. Thus, by Proposition \ref{prop4.3}(3),
$a\in N\setminus M^2$ for every maximal ideal $M$ of $D$. Hence, $N$ is not a critical maximal ideal of $D$, which finishes the proof.

(1) $\Rightarrow$ (4) Suppose that $\textnormal{char}(F)\neq 2$,
and let $a$ be a nonzero nonunit of $D$ that is not a product of finitely many prime elements of $D$.
Choose the smallest $k\in\mathbb{N}_0$ such that $a\in D_k$.
Then at least one of the prime factors of $a$ in $D_k$, say $f$, is not a prime element of $D$.
By Lemma \ref{lemma1.1}(3), there exists the smallest integer $k_1>k$ such that $f$ is not a prime element of $D_{k_1}$.
Then, in $D_{k_1}$, $f$ is a product of a nonassociated pair of prime elements by Proposition \ref{prop4.3}(2).
Hence, the number of distinct prime factors of $a$ in $D_k$ tends to infinity as $k$ grows,
which shows that $a$ is contained in infinitely many maximal ideals of $D$.

(4) $\Rightarrow$ (5) It is clear that $X-1$ is not a finite product of prime elements.
Hence, by (4), $D$ is not of finite character.

(5) $\Rightarrow$ (1) If $\textnormal{char}(F)=2$, then $D$ is of finite character by Theorem \ref{th1.5}(2).

(1) $\Rightarrow$ (6)
Since $D_n$ is a PID for each $n\in\mathbb{N}_0$, this follows directly from Proposition \ref{prop4.3}(4).

(6) $\Rightarrow$ (1) By assumption, $(X-1)D$ is a product of finitely many radical principal ideals of $D$.
If $\textnormal{char}(F)=2$, then $X-1 = (X^{\frac{1}{2^n}}-1)^{2^n}$ for each $n\in\mathbb{N}_0$,
so there exists $k\in\{1,\dots, 2^n\}$ such that $(X^{\frac{1}{2^n}}-1)^kD$ is a radical ideal of $D$,
which is a contradiction. Hence, we must have $\textnormal{char}(F)\neq 2$.
\end{proof}

Recall that a field $F$ is \textit{real closed} if $1<[\overline{F}:F]<\infty$,
where $\overline{F}$ is an algebraic closure of $F$.
It is known that if $F$ is a real closed field, then
$\textnormal{char}(F)=0$ and $[\overline{F}:F]=2$ by the Artin-Schreier theorem \cite[Corollary 8.1.15]{ka89}.
It is clear that $\mathbb{R}$, the field of real numbers, is a real closed field.
Another well-known example of real closed fields is
the algebraic closure of $\mathbb{Q}$ in $\mathbb{R}$.  We next show that $D$ has some interesting ring-theoretic properties when $F$ is a real closed field.

\begin{corollary} \label{co4}
Suppose that $F$ is either an algebraically closed field with $\textnormal{char}(F) \neq 2$ or $F$ is a real closed field.
\begin{enumerate}[font=\normalfont]
\item $D$ is an SP domain that is not Dedekind.
\item $D$ has no irreducible element.
\item Each maximal ideal of $D$ is not invertible.
\item Each nonzero nonunit of $D$ is contained in $2^{\aleph_0}$ number of maximal ideals of $D$.
\item ht$(M[[X]]/MD[[X]]) \geq 2^{\aleph_1}$ for all maximal ideals $M$ of $D$. Hence,
\begin{center}
dim$(D[[X]]) \geq 2^{\aleph_1}$,
\end{center}
and equality holds under the continuum hypothesis if $F$ is countable.
\end{enumerate}
\end{corollary}

\begin{proof}
(1) Theorem \ref{th2}.

(2)  Let $f \in D$ be a nonzero nonunit. Then $f \in D_n$ for some $n \in \mathbb{N}_0$,
and hence we may assume that $f \in F[X^{\frac{1}{2^n}}]$ with $f(0) \neq 0$.
It is clear that $f$ is not irreducible in $D_{n+2}$
because $[\overline{F}:F]\leq 2$. Thus, by Proposition \ref{coro3.2}, $f$ is not irreducible in $D$.

(3) Let $M$ be a maximal ideal of $D$. If $M$ is invertible, then $M$ is principal, say, $M=aD$
for some $a \in D$ since $D$ is a B{\'e}zout domain by Proposition \ref{prop3.1}.
Hence, $a$ is an irreducible element of $D$, which is contrary to (2) above. Thus, $M$ is not invertible.

(4) Suppose that $F$ is algebraically closed and $\textnormal{char}(F) \neq 2$. Let $f$ be a nonzero nonunit of $D$.
Since $f$ is a product of finitely many linear polynomials of $D_n$ for some $n\in\mathbb{N}_0$,
we only need to show that $X^{\frac{1}{2^n}}-\alpha$ is contained in $2^{\aleph_0}$ maximal ideals of $D$ for each nonzero $\alpha\in F$.
 Without loss of generality, we may assume that $n=0$.
Consider an infinite directed tree $\Gamma$
 whose vertices are factors of $X-\alpha$ in $D$ and two vertices $u(X),v(X)$ are connected
 by an arrow ($u(X)\to v(X)$) if and only if $u(X)=-v(X^{\frac{1}{2}})v(-X^{\frac{1}{2}})$:
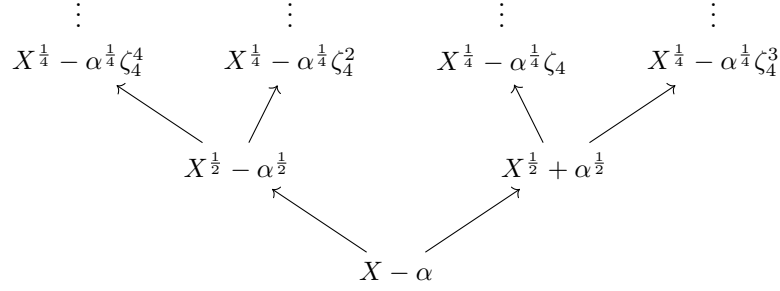
\begin{center}
    \begin{tikzpicture}[scale=.7]
 \node (5) at (-6,3) {$\vdots$};
 \node (7) at (-2,3) {$\vdots$};
 \node (9) at (2,3) {$\vdots$};
 \node (11) at (6,3) {$\vdots$};
 \node (one) at (-6,2) {$X^{\frac{1}{4}}-\alpha^{\frac{1}{4}}\zeta_4^4$};
 \node (four) at (-2,2) {$X^{\frac{1}{4}}-\alpha^{\frac{1}{4}}\zeta_4^2$};
  \node (three) at (2,2) {$X^{\frac{1}{4}}-\alpha^{\frac{1}{4}}\zeta_4$};
 \node (two) at (6,2) {$X^{\frac{1}{4}}-\alpha^{\frac{1}{4}}\zeta_4^3$};
  \node (a) at (-3,0) {$X^{\frac{1}{2}}-\alpha^{\frac{1}{2}}$};
  \node (d) at (3,0) {$X^{\frac{1}{2}}+\alpha^{\frac{1}{2}}$};
  \node (zero) at (0,-2) {$X-\alpha$};
  \draw[->] (a) -- (four);
  \draw[->] (d) -- (three);
  \draw[->] (zero) -- (a);
  \draw[->] (a) -- (one);
 \draw[->] (d) -- (two);
 \draw[->] (zero) -- (d);
\end{tikzpicture}
\end{center}
where $\alpha^{\frac{1}{4}}$ is a fourth root of $\alpha$ and $\zeta_4$ is a primitive fourth root of unity
in $F$, so $\zeta_4, \zeta_4^2, \zeta_4^3, \zeta_4^4$ are all distinct because char$(F) \neq 2$.
Given an infinite directed path $\gamma$ of $\Gamma$ starting from $X-\alpha$,
let $P(\gamma)$ be the ideal of $D$ generated by the vertices of $\gamma$.
Then $P(\gamma)$ is a maximal ideal of $R$ that contains $X-\alpha$ by Theorem \ref{th1}(1).
In fact,  by Lemma \ref{lemma2.3}, the correspondence $\gamma\mapsto P(\gamma)$ is a one-to-one correspondence from
the set of infinite directed paths of $\Gamma$ starting from $X-\alpha$
to the set of maximal ideals of $D$ containing $X-\alpha$.
Thus, $X-\alpha$ is contained in $2^{\aleph_0}$ maximal ideals of $D$. The case when $F$ is a real closed field can be derived similarly.

(5) Let $M$ be a maximal ideal of $D$. Then $M$ is countably infinitely generated
by (3) and Corollary \ref{coro2.4}. Hence, as in the proof of Corollary \ref{coro3.5}(4),
it suffices to show that $M$ is not an SFT-ideal. Assume to the contrary that $M$ is an SFT-ideal.
Then there exist a finitely generated ideal $J$ of $D$ and $k \in \mathbb{N}$
so that $a^k \in J$ for all $a \in M$. Note that $M = \bigcup\limits_{n \in \mathbb{N}_0}f_nD_n$,
where $f_n$ is an irreducible element of $R_n$, so deg$_nf_n =1$ or $2$.
Hence, $f_n$ must be factorized into at least two irreducible elements in $D_{n+2}$.
Since $J$ is finitely generated, there exists $m \in \mathbb{N}_0$ so that $J \subseteq f_mD$.
Hence, if $f_m = f_{m+2}g$ for some nonunit $g \in D_{m+2}$,
then $f_{m+2}$ and $g$ are not associated as elements of $D_{m+2}$ by Proposition \ref{prop4.3}(2)
and $f_{m+2}^k = f_mh = f_{m+2}gh$ for some $h \in D$. Note that $h \in D_{m+2}$
by Lemma \ref{lemma1.1}(4)
and $f_{m+2}$ is irreducible in $D_{m+2}$. Hence, $f_{m+2}^k = f_{m+2}gh$ implies that $1= gh_1$ for some
$h_1 \in D_{m+2}$, a contradiction. Therefore, $M$ is not an SFT-ideal.
\end{proof}

Let $A$ be a commutative ring with identity. Then $A$ is called a \textit{special primary ring} (SPR) if $A$
is a local ring with maximal ideal $M$ such that each proper ideal of $A$ is a power of $M$ or, equivalently, if
$M$ is a principal ideal and $M^n= (0)$ for some $n \in \mathbb{N}$.
A \textit{B{\'e}zout ring} is a commutative ring with identity in which each finitely generated ideal is principal.
It is clear that if $D$ is a B{\'e}zout domain, then $D/I$ is a B{\'e}zout ring for each proper ideal $I$ of $D$.
We construct a non-Noetherian B{\'e}zout ring that is locally an SPR (hence, locally Noetherian)
whose prime ideals are principal for all but one which is countably generated.

\begin{corollary}
Assume that $F = \mathbb{Q}$. Let
$m\in\mathbb{N}_0$, $M_m=\bigcup\limits_{n\in\mathbb{N}_0} \Phi_{2m+1}(X^{\frac{1}{2^n}})D_n$,
$I_m= \Phi_{2m+1}(X)D$ and $A_m = D/I_m$.
Then the following statements hold.
\begin{enumerate}[font=\normalfont]
\item $A_m$ is a zero-dimensional reduced B{\'e}zout ring.
\item Spec$(A_m) = \{\Phi_{2(2m+1)}(\frac{1}{X^{2^n}})D/I_m \mid n \in \mathbb{N}\} \cup \{M_m/I_m\}$.
\item $M_m/I_m$ is not finitely generated.
\item $(A_m)_Q$ is an SPR for all prime ideals $Q$ of $A_m$.
\end{enumerate}
\end{corollary}

\begin{proof}
By Proposition \ref{prop4.3}(4), $I_m$ is a radical ideal of $D$.
Hence, $D/I_m$ is a reduced ring. Since $D$ is a one-dimensional B{\'e}zout domain by Proposition \ref{prop3.1}(1),
the proof of Corollary \ref{excyclo} then yields (1), (2) and (3). Finally, if $Q$ is a prime ideal of $A_m$, then $Q = P/I_m$
for some prime ideal $P$ of $D$ with $I_m \subseteq P$. Then $(A_m)_Q = (D/I_m)_{P/I_m} = D_P/I_mD_P$,
and since $D_P$ is a DVR by Theorem \ref{th2}, $(A_m)_Q$ is an SPR.
\end{proof}

The final result of this section concerns the prime elements of $D$ in Notation \ref{notation}
when $F$ is a finite field of odd characteristic.
We first need a lemma which can be considered as a restatement of Proposition \ref{coro3.2} when $p=2$.

\begin{lemma} \label{lemma4.6}
For a fixed $n\in\mathbb{N}_0$, let $f(X^{\frac{1}{2^n}})$ be an irreducible element of $R_n$
with $f(0) \neq 0$.
    Then the following statements are equivalent.
    \begin{enumerate}[font=\normalfont]
       \item $f(X^{\frac{1}{2^n}})$ is a prime element of $D$.
    \item $f(X^{\frac{1}{2^n}})$ is an irreducible element of $D$.
     \item $f(X^4)$ is an irreducible element of $F[X]$.
\item $f(X^{\frac{1}{2^k}})$ is a prime element of $D$ for all $k \in \mathbb{N}_0$.
    \end{enumerate}
Hence, $D$ has no irreducible element if and only if
$f(X^4)$ is not irreducible in $F[X]$ for each irreducible polynomial $f(X)$ of $F[X]$.
\end{lemma}

\begin{proof}
$(1)\Leftrightarrow (2)$ An irreducible element of a B{\'e}zout domain is a prime element.
Thus, the result follows from Proposition \ref{prop3.1}(1).

$(2)\Leftrightarrow (3)$ By Proposition \ref{coro3.2}, $f(X^{\frac{1}{2^n}})$ is irreducible in $D$
if and only if $f(X^{\frac{1}{2^n}})$ is irreducible in $D_{n+2}$.
Note that $f(X^{\frac{1}{2^n}}) = f((X^{\frac{1}{2^{n+2}}})^4)$ in $D_{n+2}$. Thus,
$f(X^{\frac{1}{2^n}})$ is irreducible in $D$ if and only if
$f(X^4)$ is irreducible in $F[X]$.

$(3) \Leftrightarrow (4)$ This follows from the equivalence of (2) and (3) above.
\end{proof}

For an integer $n \in \mathbb{N}$, let $\Phi_n(X)$ be the $n$-th cyclotomic polynomial.
It is known that $\Phi_n(X)$ is irreducible over $\mathbb{Q}$ but
it may not be irreducible over a finite field.
For instance, $\Phi_8(X)=X^4+1$ is not irreducible over any finite field $F$.
(Proof: Let $p=\textnormal{char}(F)$. We may assume that $F=\mathbb{Z}/p\mathbb{Z}$.
Clearly, if $p=2$, then $X^4+1=(X^2+1)^2$. Now assume that $p$ is odd, so either $p\equiv 1(\textnormal{mod }4)$ or $p\equiv 3(\textnormal{mod }4)$.
If $p\equiv 1(\textnormal{mod }4)$, then $X^2+1$ has a solution in $\mathbb{Z}/p\mathbb{Z}$ \cite[Exercise 1.10]{cox},
so by Lemma  \ref{thc}(2), $X^4+1$ is not irreducible over $\mathbb{Z}/p\mathbb{Z}$. Next,
suppose that $p\equiv 3(\textnormal{mod }4)$. If $p\equiv 7(\textnormal{mod }8)$ (resp., $p\equiv 3(\textnormal{mod }8)$),
then $X^2-2$ (resp., $X^2+2$) has a nonzero solution, say $\alpha$, in $\mathbb{Z}/p\mathbb{Z}$  \cite[Exercise 1.10]{cox}.
It follows that if $p\equiv 3(\textnormal{mod }4)$, then $\alpha^{-1}$ is a solution of $4X^4-1$ in $\mathbb{Z}/p\mathbb{Z}$.
Hence, by Lemma  \ref{thc}(2) again, $X^4+1$ is not irreducible over $\mathbb{Z}/p\mathbb{Z}$.)
On the other hand,
it is well-known that every root of a polynomial over a finite field whose constant term is nonzero is a root of unity,
so it must be a factor of $X^n-1$ for some $n\in\mathbb{N}$.
Therefore, we naturally adopt the following notion: Let $f(X)\in F[X]$ be a monic polynomial with $f(0)\neq0$.
Then the \textit{order} of $f(X)$, denoted by ord$(f)$, is the smallest  $n \in \mathbb{N}$
such that $f(X)\mid X^n-1$ (note that $n\le |F|^{\textnormal{deg}(f)} -1$ \cite[Lemma 3.1]{ln97}.)

Let $\mathbb{F}_q$ be a finite field of $q$ elements and $\mathbb{F}_q^*$ the multiplicative group of
units of $\mathbb{F}_q$. Given integers $n\in\mathbb{N}$ and $b$ with $gcd(n,b)=1$,
the least positive integer $k$ such that $b^k\equiv 1(\textnormal{mod }n)$ is called the
\textit{multiplicative order} of $b$ modulo $n$. We collect some preliminary results for the
proof of Proposition \ref{lemma4.13}.

\begin{theorem} \label{Theorem 3.3}
Let $\mathbb{F}_q$ be a finite field of $q$ elements and
$f\in\mathbb{F}_q[X]$ an irreducible polynomial over $\mathbb{F}_q$ of degree $m$ 
with $f(0)\neq 0$.
\begin{enumerate} [font=\normalfont]
\item ord$(f)$ is equal to the order of any root of $f$ in $\mathbb{F}^{*}_{q^m}$.
\item ord$(f)$ divides $q^m-1$. Hence, char$(F)$ and ord$(f)$ is relatively prime.
\item $m$ is equal to the multiplicative order of $q$ modulo ord$(f)$.
\end{enumerate}
\end{theorem}

\begin{proof}
(1) \cite[Theorem 3.3]{ln97}. (2) \cite[Corollary 3.4]{ln97}.
(3) This appears in \cite[Theorem 3.5]{ln97}.
\end{proof}

\begin{theorem} \label{Theorem 2.47(ii)} \cite[Theorem 2.47(ii)]{ln97}
 Let $\mathbb{F}_q$ be a finite field of $q$ elements,
$n \in \mathbb{N}$ an integer with gcd$(q,n)=1$, and $\phi$ the Euler's totient function.
If $d$ is the multiplicative order of $q$ modulo $n$, then $\Phi_n$ factors
into $\phi(n)/d$ distinct monic irreducible polynomials in $\mathbb{F}_q[X]$ of
the same degree $d$.
\end{theorem}

\begin{theorem} \label{Theorem 3.35} Let $\mathbb{F}_q$ be a finite field of $q$ elements,
$f_1(X),\dots, f_N(X)$ distinct monic irreducible polynomials in $\mathbb{F}_q[X]$ of degree $m$ and order $e$.
\begin{enumerate} [font=\normalfont]
\item Let $t\ge 2$ be an integer whose prime factors divide $e$ but not $\frac{q^m-1}{e}$.
Assume also that $q^m\equiv 1(\textnormal{mod }4)$ if $t\equiv0(\textnormal{mod }4)$.
Then $f_1(X^t),\dots, f_N(X^t)$ are all distinct monic irreducible polynomials in $\mathbb{F}_q[X]$ of degree $mt$ and order $et$.
\item Let $q=2^au-1$ and $t=2^bv$ with $a,b\ge 2$, where $u,v$ are odd integers
and all prime factors of $t$ divides $e$ but not $\frac{q^m-1}{e}$. If $k=\min\{a,b\}$, then
each of the polynomials $f_j(X^t)$ factors as a product of $2^{k-1}$ monic irreducible polynomials in $\mathbb{F}_q[X]$ of degree $mt2^{1-k}$ and order $et$.
\end{enumerate}
\end{theorem}

\begin{proof}
(1) \cite[Theorem 3.35]{ln97}. (2) \cite[Theorem 3.37]{ln97}.
\end{proof}

We are now ready to prove the following result which completely characterizes 
an irreducible element of $D_n = F[X^{\frac{1}{2^n}}, X^{-\frac{1}{2^n}}]$ for all $n \in \mathbb{N}_0$ that is also a prime elements of $D$.

\begin{proposition} \label{lemma4.13}
Let $F$ be a finite field with $q$ elements and $\textnormal{char}(F)\neq 2$.
Suppose that $f(X) \in F[X]$ is a monic polynomial of degree $m$ and $f(0)\neq0$.
Then $f(X)$ is a prime element of $D$ if and only if the following two conditions are satisfied:
\begin{enumerate}[font=\normalfont]
    \item $4\mid q^m-1$.
    \item There exists $n\in\mathbb{N}$ such that
    \begin{enumerate}[font=\normalfont]
        \item $m=\min\{i\in\mathbb{N}\mid n\mid q^i-1\}$,
                \item $2n\nmid q^{m}-1$, and
         \item $f(X)$ is an irreducible factor of $\Phi_n(X)$ over $F$.
    \end{enumerate}
\end{enumerate}
In this case, $n$ is the order of $f(X)$, $4\mid n$ and gcd$(q, n)=1$.
Moreover, if $\textnormal{char}(F)\equiv 3$  $(\textnormal{mod }4)$, then $m$ is even.
\end{proposition}

\begin{proof}
We may assume that $f(X)$ is irreducible over $F$. Let $n\in\mathbb{N}$ be the order of $f(X)$.
Recall that if $\alpha$ is a root of $f(X)$ in an algebraic closure $\overline{F}$ 
of $F$, then $\alpha$
is a primitive $n$-th root of unity by Theorem \ref{Theorem 3.3}(1) and the multiplicity of $\alpha$ is $1$
(this is obvious by the facts that $f(X)$ divides $X^n-1$
 and $X^n-1$ does not have any repeated roots since $n$ is relatively prime to $q$
 by Theorem \ref{Theorem 3.3}(2).
Therefore, $f(X)$ divides $\Phi_n(X)$ in $\overline{F}[X]$,
so $f(X)$ divides $\Phi_n(X)$ in $F[X]$ by Example \ref{ex2.12}(3) and (4).
It also follows that $m=\min\{i\in\mathbb{N}\mid n\mid q^i-1\}$ by Theorem \ref{Theorem 3.3}(3).
Thus, we only need to show that $f(X)$ is a prime element of $D$ if and only if $4\mid q^m-1$ and $2n\nmid q^m-1$.

$(\Rightarrow)$ Assume that $f(X)$ is a prime element of $D$. Then $f(X^4)$ is irreducible in $F[X]$ by Lemma \ref{lemma4.6}.
Notice that $f(X^4)$ is a factor of $\Phi_{n}(X^4)$.
If $n$ is odd, then  $\Phi_{n}(X^4)=\Phi_{4n}(X)\Phi_{2n}(X)\Phi_{n}(X)$
as mentioned in the proof of Corollary \ref{excyclo}.
Note that both $2$ and $n$ divide $q^m-1$ and $n$ is odd, so $2n\mid q^m-1$, which implies that $m=\min\{i\in\mathbb{N}\mid 2n\mid q^i-1\}$
because $m=\min\{i\in\mathbb{N}\mid n\mid q^i-1\}$.
Therefore, the degree of irreducible factors of $\Phi_{n}(X)$ and $\Phi_{2n}(X)$ over $F$ must be equal to $m$  by Theorem \ref{Theorem 2.47(ii)}.
It follows that $f(X^4)$ is an irreducible factor of $\Phi_{4n}(X)$, so the order of $f(X^4)$ must be $4n$ by Theorem \ref{Theorem 3.35}(1)
and $4m=\min\{i\in\mathbb{N}\mid 4n\mid q^i-1\}$ by Theorem \ref{Theorem 3.3}(3).
However, since $2|q^m-1$, $n|q^m-1$ and $2|q^m+1$, we have $4n|q^{2m}-1$, a contradiction.
Hence, $n$ must be even. Then $\Phi_{n}(X^4)=\Phi_{4n}(X)$ as mentioned in the proof of Corollary \ref{excyclo}.
In this case, $f(X^4)$ is an irreducible factor of $\Phi_{4n}(X)$, so its roots are all $4n$-th primitive root of unity,
which implies that $f(X^4)$ is of order $4n$.
Hence, if $2n\mid q^m-1$, then $4n\mid q^{2m}-1$,
 a contradiction. Thus, $2n\nmid q^m-1$.
It remains to show that $4\mid q^m-1$. If $4\nmid q^m-1$, then we have $q\equiv -1(\textnormal{mod }4)$ and $m$ is odd.
 Hence, by Theorem \ref{Theorem 3.35}(2), $f(X^4)$ is factored as a product of at least $2$ monic irreducible polynomials, a contradiction.

($\Leftarrow)$ Suppose that $4\mid q^m-1$ and $2n\nmid q^m-1$. Then $n$ must be even, since otherwise we would have $2n\mid q^m-1$.
Hence, $f(X^4)$ is a monic irreducible polynomial of $F[X]$ by Theorem \ref{Theorem 3.35}(1).
Thus, $f(X)$ is a prime element of $D$ by Lemma \ref{lemma4.6}.

\vspace{.2cm}

 In this case, assume that char$(F)=p$. Then $q=p^k$ for some $k \in \mathbb{N}$ and gcd$(p,n)=1$, thus gcd$(q,n)=1$.
Also, $4|n$ because $4|q^m-1$, $n|q^m-1$ and $2n\nmid q^m-1$.
 Moreover, if $p \equiv 3$  $(\textnormal{mod }4)$, then $4|q^m-1$ implies that $m$ is even.
\end{proof}

Note that every prime element of $D$ is of the form $uf(X^\frac{1}{2^n})$ for some nonzero $u\in F$
and a monic irreducible polynomial $f(X)\in F[X]$ that is a prime element of $D$. Hence, if $F$ is a
finite field of char$(F) \neq 2$, then we can find every prime element of $D$ using Proposition \ref{lemma4.13}.
It is also well known that if $|F|=q < \infty$ and $N_q(m)$ denotes the number of monic irreducible polynomials of degree $m$ in $F[X]$,
then $$N_q(m) = \frac{1}{m}\sum\limits_{d|m}\mu(d)q^{\frac{m}{d}}$$
where $\mu(d)$ is the M\"obius function \cite[Theorem 3.25]{ln97}.
In the next corollary, we find how many of these polynomials are prime elements of $D$.

\begin{corollary} \label{coro4.8}
Let $F$ be a finite field with $q$ elements, $\textnormal{char}(F)\neq 2$,
$m \in \mathbb{N}$, and $\phi$ the Euler's totient function.
Then the number of monic polynomials of degree $m$ in $F[X]$ that are prime elements of $D$ is
\[\begin{cases}
    0 & \textnormal{ if }4\nmid q^m-1,\\
    &\\
    \sum\limits_{n}\frac{\phi(n)}{m} & \textnormal{ if }4\mid q^m-1
\end{cases}\]
where the sum is taken over all $n\in\mathbb{N}$ satisfying \textnormal{(a)} and \textnormal{(b)} of Proposition \ref{lemma4.13}.
\end{corollary}

\begin{proof}
This follows from Proposition \ref{lemma4.13} by recalling that for each $n\in\mathbb{N}$ satisfying \textnormal{(a)}
and \textnormal{(b)}, $\Phi_n(X)$ is factored into $\frac{\phi(n)}{m}$ distinct monic irreducible polynomials of degree $m$  by Theorem \ref{Theorem 2.47(ii)}.
 \end{proof}

Let $\mathbb{F}_q$ be a finite field with $q$ elements and $\mathbb{F}_q^*$
the multiplicative group of units of $\mathbb{F}_q$, so $\mathbb{F}_q^*$ is a cyclic group.
An element $\alpha\in \mathbb{F}_{q}^*$ is \textit{primitive} if it is a group generator of $\mathbb{F}_{q}^*$.
A polynomial $f(X)\in\mathbb{F}_q[X]$ is \textit{primitive over }$\mathbb{F}_q$ if there exists
a primitive element $\alpha$ of $\mathbb{F}_{q^m}^*$ for some $m\in\mathbb{N}$ such that $f(X)$ is
a minimal polynomial of $\alpha$ over $\mathbb{F}_q$.

\begin{corollary}
Assume that $F$ is a finite field with $q$ elements.
If $m$ is a positive integer such that $4\mid q^m-1$, then each primitive polynomial of degree $m$ over $F$ is a prime element of $D$.
\end{corollary}

\begin{proof}
Recall that  a degree $m$ polynomial over $F$ is primitive
if and only if it is a monic polynomial with nonzero constant term whose order is $q^m-1$ \cite[Theorem 3.16]{ln97}.
Thus, Proposition \ref{lemma4.13} yields the conclusion.
\end{proof}

It is helpful to note that if
$m$ is an integer with $m>1$, then the product of all monic irreducible polynomials of degree $m$
in $F[X]$ equals $\prod\limits_{n}\Phi_n(X)$, where $n$ is taken over all positive integers
that satisfy (a) of Proposition \ref{lemma4.13} \cite[Theorem 3.31]{ln97}.
 The result of Theorem \ref{Theorem 2.47(ii)} also gives a very helpful information 
when we factor $\Phi_n(X)$ into irreducible polynomials.
We end this section with an example in which we classify all prime elements of
$D$ that are monic polynomials of degree $4$ in $F[X]$ with $|F|=3$.

\begin{example} \label{ex4.10}
{\em Let $F$ be a field with three elements. We will identify every monic polynomial of degree 4 in $F[X]$
that is a prime element of $D$. Note that $q^m-1=3^4-1=80$. One readily verifies that the only $n\in\mathbb{N}$
satisfying  \textnormal{(a)}  of Proposition \ref{lemma4.13} are $n=5,10,16,20,40, 80$.
Moreover, $N_3(4) = \frac{1}{4}\sum\limits_{d|4}\mu(d)3^{\frac{4}{d}}
= \frac{1}{4}(3^4\mu(1) + 3^2\mu(2) + 3\mu(4)) = 18$. Hence,
there are $18$ monic irreducible polynomials of degree $4$ in $F[X]$. We will show  that only 10 of them are prime elements of $D$.
With the aid of Macaulay2 \cite{gs}, we have the following prime factorization over $F$:
\begin{align*}
\Phi_5(X)=&X^4+X^3+X^2+X+1,\\
\Phi_{10}(X)=&X^4-X^3+X^2-X+1,\\
    \Phi_{16}(X)=&X^8+1
     =(X^4+X^2+2)(X^4+2X^2+2),\\
    \Phi_{20}(X)=&X^8-X^6+X^4-X^2+1\\
    =&(X^4+X^3+2X^2+1)(X^4+2X^2+X+1),\\
        \Phi_{40}(X)=&X^{16}-X^{12}+X^8-X^4+1\\
    =&(X^4+X^2+X+1)(X^4+X^2+2X+1)\\
    &(X^4+X^3+X^2+1)(X^4+2X^3+X^2+1)\\
    \Phi_{80}(X)=&X^{32} - X^{24} + X^{16} - X^8 + 1\\
    =&(X^4 + X + 2) (X^4 + 2 X + 2) (X^4 + X^3 + 2)\\
    &(X^4 + X^3 + X^2 + 2 X + 2) (X^4 + X^3 + 2 X^2 + 2 X + 2) (X^4 + 2 X^3 + 2)\\
    &(X^4 + 2 X^3 + X^2 + X + 2) (X^4 + 2 X^3 + 2 X^2 + X + 2).
\end{align*}
Hence, $\{X^4+X^3+X^2+X+1, X^4-X^3+X^2-X+1, X^4+X^2+2, X^4+2X^2+2, X^4+X^3+2X^2+1, X^4+2X^2+X+1, X^4+X^2+X+1, X^4+X^2+2X+1,
X^4+X^3+X^2+1, X^4+2X^3+X^2+1, X^4 + X + 2, X^4 + 2 X + 2, X^4 + X^3 + 2,
X^4 + X^3 + X^2 + 2 X + 2, X^4 + X^3 + 2 X^2 + 2 X + 2, X^4 + 2 X^3 + 2, X^4 + 2 X^3 + X^2 + X + 2, X^4 + 2 X^3 + 2 X^2 + X + 2\}$ is
the set of all monic irreducible polynomials of degree 4 over $F$.

On the other hand, the only $n\in\mathbb{N}$
satisfying \textnormal{(a)} and \textnormal{(b)} of Proposition \ref{lemma4.13} are $n=16, 80$.
Hence, by Proposition \ref{lemma4.13}, $\{X^4+X^2+2, X^4+2X^2+2, X^4 + X + 2, X^4 + 2 X + 2, X^4 + X^3 + 2,
X^4 + X^3 + X^2 + 2 X + 2, X^4 + X^3 + 2 X^2 + 2 X + 2, X^4 + 2 X^3 + 2, X^4 + 2 X^3 + X^2 + X + 2, X^4 + 2 X^3 + 2 X^2 + X + 2\}$
is the set of all monic polynomials of degree $4$ in $F[X]$ that are prime elements of $D$.

Note that $X^4+X^2+2$ and $X^4+2X^2+2$ are the only non-primitive elements of this set.}
\end{example}


\section{The ring $D= \bigcup\limits_{n\in\mathbb{N}_0}\mathbb{Q}[X^{\frac{1}{2^n}}, X^{-\frac{1}{2^n}}]$}

Let $D=\bigcup\limits_{n\in\mathbb{N}_0} F[X^{\frac{1}{p^n}}, X^{-\frac{1}{p^n}}]$ for a prime number $p$. In the previous sections, a cyclotomic polynomial plays a central role when we seek element-wise factorization properties of $D$ in terms of those of $F[X^{\frac{1}{p^n}}, X^{-\frac{1}{p^n}}]$.
However, a cyclotomic polynomial is not irreducible in general (see
Example \ref{ex4.10}), so in order to ensure the irreducibility of a cyclotomic polynomial
in $F[X^{\frac{1}{p^n}}, X^{-\frac{1}{p^n}}]$,
in this section, we choose $F=\mathbb{Q}$. We also need the results developed in Section 4, so we maintain our assumption that $p=2$.

Let $D= \bigcup\limits_{n\in\mathbb{N}_0}\mathbb{Q}[X^{\frac{1}{2^n}}, X^{-\frac{1}{2^n}}]$.
Then $X^{\frac{1}{2^n}} +1$ is a prime element of $D$ by Lemma \ref{lemma4.6}
and $X^{\frac{1}{2^n}} +1$ divides $X-1$ in $D$ for all $n \in \mathbb{N}$.
In fact, we have
\begin{align*}
    X-1=(X^{\frac{1}{2}}-1)&(X^{\frac{1}{2}}+1)\\
=(X^{\frac{1}{4}}-1)&(X^{\frac{1}{4}}+1)(X^{\frac{1}{2}}+1)\\
    &\vdots\\
=(X^{\frac{1}{2^n}}-1)&\prod\limits_{i=1}^{n}(X^{\frac{1}{2^i}}+1)\\
&\vdots
\end{align*}
for all $n \in \mathbb{N}$. So we want to write $X-1 = \prod\limits_{n \in \mathbb{N}}(X^{\frac{1}{2^n}} +1)$.
In fact, in this section, we are going to prove that every nonzero nonunit
of $D$ can be written as a product of countably many prime elements of $D$ (Theorem \ref{theorem5.11}).
We first need a notion of infinite product of prime elements in $D$.

\begin{definition} \label{defi5.2}
{\em Let the notation be as in Notation \ref{notation} and
$S$ the set of monic polynomials of $R$ that are prime elements of $D$.
For a nonzero nonunit $f$ of $D$, assume that for each $p\in S$,
\begin{enumerate}[font=\normalfont]
\item  there exists $k_p \in \mathbb{N}_0$ such that $p^{k_p}|f$ but $p^{k_p+1}\nmid f$ and
\item if $g \in D$ is such that $p^{k_p}|g$ but $p^{k_p+1}\nmid g$ for each $p\in S$, then $gD \subseteq fD$.
\end{enumerate}
Then $f$ will be written as a product of prime elements in $S$, i.e., $f = u_f\prod\limits_{p \in S}p^{k_p}$,
where $u_f$ is the unit of $D$ such that $f=u_ff_1$ for a monic polynomial $f_1$ of $R$ with $f_1(0) \neq 0$.
Moreover, if $T = \{p \in S\mid k_p \neq 0\}$,
then $f$ is also written as $f = u_f\prod\limits_{p \in T}p^{k_p}$, and $f$ is a unit of $D$ if and only if $S = T$. }
\end{definition}

It is worth noting that the (infinite) product of prime elements in Definition \ref{defi5.2} is unique because
each prime element of $D$ is associated with a unique monic polynomial in $F[X^{\frac{1}{2^n}}]$ for some $n \in \mathbb{N}_0$
and that if $F$ is countable (e.g., $F$ is finite or $F= \mathbb{Q}$), then the set $S$ of Definition \ref{defi5.2} is countable.

For $d\in\mathbb{N}$,
let $\Phi_d(X) \in \mathbb{Z}[X]$ be the $d$-th cyclotomic polynomial.
It is useful to note that $\Phi_d(X)$ is divided by (countably) infinitely  many prime elements in
$D= \bigcup\limits_{n\in\mathbb{N}_0}\mathbb{Q}[X^{\frac{1}{2^n}}, X^{-\frac{1}{2^n}}]$
if and only if $d$ is odd, as the following lemma shows.

\begin{lemma} \label{lemma5.1}
Let the notation be as in Notation \ref{notation} with $F=\mathbb{Q}$.
 Then the following statements hold for each $d\in\mathbb{N}$.
\begin{enumerate}[font=\normalfont]
\item $d$ is even if and only if $\Phi_d(X)$ is a prime element of $D$.
\item If $d$ is odd, then $\Phi_d(X)$ is contained in countably infinitely many maximal ideals of $D$, say, $\{N\}\cup\{\Phi_{2d}(X^{\frac{1}{2^i}})D\mid i\in\mathbb{N}\}$,
where $N$ is infinitely generated by $\{\Phi_d(X^{\frac{1}{2^i}})\mid i\in\mathbb{N}\}$. 
\end{enumerate}
\end{lemma}

\begin{proof}
This is a special case of Corollary \ref{excyclo}. In particular,
(2) appears in the proof of Corollary \ref{excyclo}(2).
\end{proof}

We will show that each nonzero nonunit of
$D= \bigcup\limits_{n\in\mathbb{N}_0}\mathbb{Q}[X^{\frac{1}{2^n}}, X^{-\frac{1}{2^n}}]$
is contained in countably many maximal ideals of $D$, in contrast to Corollary \ref{co4}(4),
which is a simple corollary of the following theorem.
For the proof of the next result, we borrow a technique from that of \cite[Theorem 4.6]{bglz24}.

\begin{theorem} \label{q1}
Let the notation be as in Notation \ref{notation} with $F=\mathbb{Q}$ and
$f(X)\in\mathbb{Q}[X]$ a nonconstant polynomial with $f(0)\neq 0$. If $f(X)$ is irreducible over $\mathbb{Q}$,
then the following statements are equivalent.
\begin{enumerate}[font=\normalfont]
\item $f(X)=u\Phi_d(X)$ for some nonzero $u\in\mathbb{Q}$ and an odd integer $d\in\mathbb{N}$.
\item $f(X)$ is not a product of finitely many prime elements of $D$.
\end{enumerate}
\end{theorem}

\begin{proof}
(1) $\Rightarrow$ (2) This follows from Lemma \ref{lemma5.1}
and the fact that dim$(D)=1$.

(2) $\Rightarrow$ (1) Assume that (2) holds. We need to show that $f(X)$ is associated to $\Phi_d(X)$ for some odd integer $d$.
Note that $f(X)$ is associated to some $f_0(X)\in\mathbb{Z}[X]$ with $c(f_0(X)) = \mathbb{Z}$.
We may therefore assume that $f(X)=f_0(X)$.
Then $f_0(X)$ is not a product of finitely many prime elements of $D$ by assumption.

By Proposition \ref{coro3.2}, $f_0(X)$ is not irreducible in $D_2$. Hence,
either $f_0(X)=u_1f_1(X^{\frac{1}{2}})f_1(-X^{\frac{1}{2}})$ or $f_0(X)=u_1f_1(X^{\frac{1}{4}})f_1(-X^{\frac{1}{4}})$
for some $0 \neq u_1\in\mathbb{Q}$ and an irreducible polynomial $f_1(X)\in\mathbb{Q}[X]$ by Proposition \ref{prop4.3}(2).
Then, by Gauss's lemma, $f_1(X)\in\mathbb{Z}[X]$, $u_1=\pm1$ and $c(f_1(X)) = \mathbb{Z}$.
Since $f_0(X)$ is not a product of finitely many prime elements of $D$, without loss of generality,
we may assume that $f_1(X)$ is not a prime element of $D$. Hence, applying the preceding argument,
we have either $f_1(X)=\pm f_2(X^{\frac{1}{2}})f_2(-X^{\frac{1}{2}})$ or $f_1(X)=\pm f_2(X^{\frac{1}{4}})f_2(-X^{\frac{1}{4}})$
for some irreducible polynomial $f_2(X)\in\mathbb{Z}[X]$ with $c(f_2(X)) = \mathbb{Z}$. Iterating this process, we have a sequence $\{f_n(X)\}_{n\in\mathbb{N}_0}$
of irreducible polynomials in $\mathbb{Z}[X]$ such that if
$b_n$ is the leading coefficient of $f_n(X)$ for every $n\in\mathbb{N}_0$, then
$b_{n}=\pm b_{n+1}^2$. Hence, $b_0=\pm 1$,
so we may assume that $f_0(X)$ is a monic irreducible polynomial in $\mathbb{Z}[X]$. \\
\\
Claim 1: $\lim\limits_{n\to\infty}\textnormal{deg}(f_n(X))\in\mathbb{N}$.\\
\\
Proof of Claim 1: Let $K$ be the field obtained by adjoining $\sqrt{-1}$ to a splitting field of $f_0(X)$ over $\mathbb{Q}$.
For convenience, we first introduce an \textit{ad hoc} term: $g(X)\in\mathbb{Q}[X]$ is \textit{1-irreducible}
(resp., \textit{2-irreducible}) if it satisfies the following three conditions.
\begin{enumerate}
    \item[$\cdot$] $g(X)$ is irreducible over $\mathbb{Q}$.
    \item[$\cdot$] $g(X^2)$ is irreducible (resp., not irreducible) over $\mathbb{Q}$.
    \item[$\cdot$] $g(X^4)$ is not irreducible over $\mathbb{Q}$.
\end{enumerate}
Fix $n\in\mathbb{N}_0$. Let $\beta$ be a root of $f_n(X)$ such that $\beta \in K$. If $f_n(X)$ is 1-irreducible,
then $-\frac{\beta}{4}=\gamma^4$ for some $\gamma\in\mathbb{Q}(\beta)$ by Lemma \ref{thc}.
In this case, $\pm \beta^{\frac{1}{4}}$ is an element of $K$ because $(-4)^{\frac{1}{4}}=1+\sqrt{-1}\in K$.
Since $f_n(X)$ is 1-irreducible, it follows that $f_n(X)=\pm f_{n+1}(X^{\frac{1}{4}})f_{n+1}(-X^{\frac{1}{4}})$, and
hence either $\beta^{\frac{1}{4}}$ or $-\beta^{\frac{1}{4}}$ is a root of $f_{n+1}(X)$.
By a similar reasoning, if $f_n(X)$ is 2-irreducible, then $f_n(X)=\pm f_{n+1}(X^{\frac{1}{2}})f_{n+1}(-X^{\frac{1}{2}})$,
so $f_{n+1}(\beta^{\frac{1}{2}})f_{n+1}(-\beta^{\frac{1}{2}})=0$ and
$\pm\beta^{\frac{1}{2}} = \delta$ for some $\delta\in\mathbb{Q}(\beta)\subseteq K$.
In other words, one of the following holds:
\begin{enumerate}[label=(\roman*)]
\item $\beta^{\frac{1}{4}}\in K$ and $f_{n+1}(X)$ has either $\beta^{\frac{1}{4}}$
or $- \beta^{\frac{1}{4}}$ as one of its roots. In this case, $\textnormal{deg}(f_{n+1}(X))=2\textnormal{deg}(f_n(X))$.
\item $\beta^{\frac{1}{2}}\in K$ and $f_{n+1}(X)$ has either $\beta^{\frac{1}{2}}$ or $- \beta^{\frac{1}{2}}$ as one of its roots.
In this case, $\textnormal{deg}(f_{n+1}(X))=\textnormal{deg}(f_n(X))$.
\end{enumerate}
\noindent
Therefore, by mathematical induction, if $\alpha$ is a root of $f_0(X)$,
there exists $k_n\in\mathbb{N}$
so that at least one of $2^{k_n}$-th root of $\alpha$, say $\alpha_n$, is a root of $f_{n}(X)$ and $\alpha_n \in K$ for each $n\in\mathbb{N}$.
Since $f_{n}(X)$ is irreducible over $\mathbb{Q}$,
it follows that $\textnormal{deg}(f_{n}(X)) \leq [K:\mathbb{Q}] < \infty$. Since
$\{\textnormal{deg}(f_n(X))\}_{n\in\mathbb{N}_0}$ is a monotone increasing sequence in $\mathbb{N}$,
Claim 1 is proved. \\
\\
Now, by Claim 1, there exists $m\in\mathbb{N}$ such that $\textnormal{deg}(f_n(X))=\textnormal{deg}(f_m(X))$ for each $n\ge m$, and thereby $f_{n}(X)=\pm f_{n+1}(X^{\frac{1}{2}})f_{n+1}(-X^{\frac{1}{2}})$ for each $n\ge m$.\\
\\
Claim 2: $f_{n}(X)=f_{n+a}(X)$ for some $n\ge m$ and $a\in\mathbb{N}$.\\
\\
Proof of Claim 2: Let $k=\textnormal{deg}(f_m(X))$ and choose $n\ge m$.
Then $\textnormal{deg}(f_n(X))=k$. Write $f_n(X)=X^k+a_{n,k-1}X^{k-1}+\cdots+a_{n,0}$ for some $a_{n,0},\dots, a_{n,k-1}\in\mathbb{Z}$.
Let $r_{n,1},\dots, r_{n,k}$ be the complex roots of $f_{n}(X)$. Then, by Vieta's formula, we have
\begin{align*}
(-1)^la_{n,k-l}&=\sum\limits_{1\le i_1<i_2<\cdots <i_l\le k}\bigg(\prod\limits_{t=1}^l r_{n,i_t}\bigg)
\end{align*}
for each $l\in\{1,\dots, k\}$. Let  $\delta_n=\max\limits_{1\le i\le k}\{|r_{n,i}|\}$ and $s_n={k\choose \lceil \frac{k}{2}\rceil} \delta_n^k$.
Since $f_m(0)\neq 0$ and $f_m(X)\in\mathbb{Z}[X]$, we have $\delta_m\ge 1$. Moreover, $\delta_{n+1}=\delta_{n}^{\frac{1}{2}}$ for each $n\ge m$.
It follows that $|a_{n,k-l}|\le {k\choose l} \delta_n^l\le s_n\le s_m$ for each $l\in\{1,\dots, k\}$, i.e.,
the absolute value of a coefficient of $f_n(X)$ cannot exceed  $s_m$ for every $n\ge m$.
Hence, one can deduce that $\{f_n(X)\}_{n\ge m}$ is finite as a set. Hence, Claim 2 follows.

Now, $f_n(X)=f_{n+a}(X)$ for some $a\in\mathbb{N}$ and $n\ge m$ by Claim 2.
Note that $f_n(X) = f_{n+a}(X^{\frac{1}{2^a}})g(X^{\frac{1}{2^a}})$ for some
$g(X^{\frac{1}{2^a}}) \in \mathbb{Q}[X^{\frac{1}{2^a}}]$, so if
$\{\beta_1,\dots, \beta_k\}$ is the set of roots of $f_n(X)$, then $\{\beta_1^{2^a},\dots, \beta_k^{2^a}\}$ is also
a set of roots of $f_n(X)$. Hence, $\beta_1^{2^{a+i}}$ is a root of $f_n(X)$
for all $i \in \mathbb{N}$, and thus
$\beta_1^{2^{i}}= \beta_1^{2^j}$ for some $i, j\in\mathbb{N}$ with $i < j$.
Then $\beta_1\neq 0$ implies that $\beta_1$ is a root of unity.
Finally, note that $f_0(X) = f_{n}(X^{\frac{1}{2^c}})h(X^{\frac{1}{2^c}})$ for some
$h(X^{\frac{1}{2^c}}) \in \mathbb{Q}[X^{\frac{1}{2^c}}]$ with $c \in \mathbb{N}_0$. Hence,
$f_0(\beta_1^{2^c}) = f_n(\beta_1)h(\beta_1)=0$ and $\beta_1^{2^c}$ is a root of unity,
so it follows that $f_0(X)$ is a monic irreducible polynomial of $\mathbb{Z}[X]$
that has a root of unity as one of its roots.
Obviously, every root of unity is a primitive $d$-th root of unity for some $d\in\mathbb{N}$.
Hence, $f_0(X)$ is the minimal polynomial of a primitive $d$-th root of unity over $\mathbb{Q}$ for some $d\in\mathbb{N}$.
Thus, $f_0(X)=\Phi_d(X)$. By Lemma \ref{lemma5.1}, we conclude that $d$ is odd.
\end{proof}

\begin{corollary}
Let the notation be as in Notation \ref{notation} with $F=\mathbb{Q}$.
 Then each nonzero nonunit of $D$ is contained in  countably many maximal ideals of $D$.
\end{corollary}

\begin{proof}
Given $n\in\mathbb{N}_0$, each irreducible polynomial of $D_n$ is contained in countably many maximal ideals of $D$
by Lemma \ref{lemma5.1} and Theorem \ref{q1}. Since each nonzero nonunit of $D$ is a finite product of
irreducible polynomials of $D_n$ for some $n\in\mathbb{N}_0$, the conclusion follows.
\end{proof}

 It is remarkable that the results of Theorem \ref{q1} and Lemma \ref{lemma5.1}
also implies that each nonzero nonunit of $D$
has a prime divisor.

\begin{proposition} \label{lemma5.5}
Let the notation be as in Definition \ref{defi5.2}.
\begin{enumerate}[font=\normalfont]
\item $D = \bigcap\limits_{p \in S}D_{pD}$.
\item If $f,g \in D$ are nonzero nonunits such that
$f = u_f\prod\limits_{p \in S}p^{k_p}$ and $g = u_g\prod\limits_{p \in S}p^{l_p}$ for some units $u_f, u_g \in D$,
then $fg=u_fu_g\prod\limits_{p \in S}p^{k_p+l_p}$ and $u_{fg}=u_fu_g$.
\end{enumerate}
\end{proposition}

\begin{proof}
(1) Let $\theta \in \bigcap\limits_{p \in S}D_{pD}$, then $\theta = \frac{f}{g}$ for some
$f, g \in D$ with $(f,g) = D$. Suppose that $g$ is not a unit of $D$.
Then there exists $p \in S$ so that $p|g$ by Theorem \ref{q1} and Lemma \ref{lemma5.1}.
However, note that $\frac{f}{g} \in D_{pD}$ implies $sf \in gD \subseteq pD$ for
some $s \in D \setminus pD$ by which we have $(f,g) \subseteq pD \subsetneq D$, a contradiction. Thuu, $g$ is a unit of $D$ and $\theta = \frac{f}{g} \in D$.

(2) Let $h \in D$ be an element such that $p^{k_p+l_p}|h$ but $p^{k_p+l_p+1}\nmid h$ for each $p\in S$.
Then $h\in gD$ by (1) above, and hence $h=gt$ for some $t\in D$. It follows that $p^{k_p}|t$ but $p^{k_p+1}\nmid t$ for each $p\in S$.
Indeed, note that $p^{k_p}\mid g$ and $p^{k_p+1}\nmid g$, so $g=p^{k_p}g'$ for some $g'\in D$ such that $p\nmid g'$.
Note also that $D$ has Krull dimension 1
by Proposition \ref{prop3.1}(1), so $p^{k_p}D$ is a $pD$-primary ideal of $D$ \cite[Proposition 4.2]{am69}.
Hence, $p^{k_p+l_p}\mid h=gt$ implies $p^{k_p}\mid g't$, and thereby $p^{k_p}\mid t$.
On the other hand, $p^{k_p+1}\nmid t$ since $p^{k_p+l_p+1}\nmid gt=h$. Therefore, $t\in fD$, which implies that $h\in fgD$.
Hence, $fg=u_{fg}\prod\limits_{p\in S}p^{k_p+l_p}$. Finally, it is clear that $u_{fg}=u_fu_g$. Thus,
$fg=u_{f}u_g\prod\limits_{p\in S}p^{k_p+l_p}$.
\end{proof}

\begin{corollary} \label{inf}
Let the notation be as in Notation \ref{notation} with $F=\mathbb{Q}$, $k_1,\dots, k_a$
positive integers and $l_1,\dots, l_a$ distinct positive odd integers.
Then we have \begin{align*}
    \Phi_{l_1}(X)^{k_1}\cdots \Phi_{l_a}(X)^{k_a}=\prod\limits_{n \in \mathbb{N},\atop j\in\{1,\dots, a\}}\Phi_{2l_j}(X^{\frac{1}{2^n}})^{k_j}.
\end{align*}
\end{corollary}

\begin{proof}
By Proposition \ref{lemma5.5}, we may assume that $a=1$. Let $k_1=k$, $l_1=l$ and $f=\Phi_{l}(X)^{k}$ for the sake of simplicity.
Given $n\in\mathbb{N}$, let $\alpha_{n} =\Phi_{2l}(X^{\frac{1}{2^n}})$.
Then $\{\alpha_{n}\}_{n\in\mathbb{N}}$ is a set of monic polynomials of $R$ that are
pairwise nonassociated prime elements of $D$ and
\begin{align*}
f=\Phi_{l}(X^{\frac{1}{2^{t+1}}})^{k}\prod\limits_{n=1}^t\alpha_{n}^{k} \tag{a}
\end{align*}
 for each $t\in\mathbb{N}$ by the proof of Corollary \ref{excyclo}.
 It follows that $\alpha_{n}^{k}| f$ but $\alpha_{n}^{k+1}\nmid f$ for each $n \in \mathbb{N}$ by Proposition \ref{prop4.3}(3).
Now assume that there exists $g \in D$ such that $\alpha_{n}^{k}|g$ but $\alpha_{n}^{k+1}\nmid g$ for all $n \in \mathbb{N}$.
Then $g \in D_m$ for some $m \in \mathbb{N}$, and hence $$g= g_1 \cdots g_s \prod\limits_{n=1}^m\alpha_{n}^{k}$$
for some irreducible elements $g_1, \dots , g_s$ of $D_m$. Note that
$\alpha_{m+1}^{k}|g_1\cdots g_{s}$ in $D$.
Note also that $\alpha_{m+1}^2\nmid g_i$ for each $i\in\{1,\dots, s\}$ by Proposition \ref{prop4.3}(3).
Hence, $s \geq k$ and we may assume that $\alpha_{m+1}\mid g_{i}$ for each $i\in\{1,\dots, k\}$.
For each $i\in\{1,\dots, k\}$, we then have
$g_{i} =u_i\alpha_{m+1}\Phi_{2l}(-X^{\frac{1}{2^{m+1}}})=u_i\alpha_{m+1}\Phi_{l}(X^{\frac{1}{2^{m+1}}})$
for some nonzero $u_i\in\mathbb{Q}$,
where the first equality follows from Proposition \ref{prop4.3}(2) and the second is from \cite[p.280]{l02}.
Thus, $g \in fD$ or $gD \subseteq fD$ by equation (a) above. Therefore,
the statement follows.
\end{proof}

We are now ready to prove the main result of this section.

\begin{theorem} \label{theorem5.11}
Let the notation be as in Notation \ref{notation} with $F=\mathbb{Q}$.
Then each nonzero nonunit of $D$ can be written uniquely as a product of countably many prime elements of $D$.
\end{theorem}

\begin{proof}
Let $S$ be the set of monic polynomials of $R$ that are prime elements of $D$.
Then $S$ is countable.
Now let $f=u_f\prod\limits_{p\in S}p^{k_p}$ be the product of prime elements in $D$ for a nonzero nonunit $f$ of $D$.
It is clear that if $q \in S$, then $qf=u_f\prod\limits_{p\in S}p^{l_p} \in D$ where
\[l_p=\begin{cases}
 k_p &\textnormal{ if }p\neq q,\\
k_q+1 &\textnormal{ if }p=q.
\end{cases}\]
Hence,  by Corollary \ref{inf} and Theorem \ref{q1}, the conclusion follows.
\end{proof}

Let the notation be as in Theorem \ref{theorem5.11} and
$f \in D$ a nonzero nounit. Then there exists a sequence $\{k_p\}_{p \in S}$ of nonnegative integers
such that $f = u_f\prod\limits_{p\in S}p^{k_p}$ for some unit $u_f$ of $D$.
We next classify all the
sequences of nonnegative integers that represent elements of $D$,
i.e., we determine when $\prod\limits_{p\in S}p^{k_p} \in D$ for a given sequence $\{k_p\}_{p \in S}$ of nonnegative integers.

\begin{proposition} \label{lastcoro}
Let the notation be as in Notation \ref{notation} with $F=\mathbb{Q}$,
$S$ the set of monic polynomials of $R$ that are prime elements of $D$ and
$\{a_p\}_{p \in S}$ a sequence of nonnegative integers such that $T: = \{p \in S \mid a_p > 0\}$ is infinite.
Then $\prod\limits_{p \in S}p^{a_p} \in D$ if and only if the following two conditions hold:
\begin{enumerate}[font=\normalfont]
    \item There exist a finite subset $\mathcal{C}$ of $T$, a nonnegative integer $n$ and distinct positive odd integers $d_1,\dots, d_l$ such that $$T \setminus\mathcal{C}=\bigcup\limits_{j=1}^l\{\Phi_{2d_j}(X^{\frac{1}{2^i}})\mid i \gneq n\}.$$
    \item For each $j\in\{1,\dots, l\}$, let
    \begin{align*}
        \mathcal{B}_j&=\{p \in T \mid p =\Phi_{2d_j}(X^{\frac{1}{2^i}})\textnormal{ for some }i\ge n\},\\
        b_j&=\min\{a_p\mid p\in B_j\}.
    \end{align*}
    Then $a_p=b_j$ for all but finitely many $p \in \mathcal{B}_j$.
\end{enumerate}
In this case, $\{a_p\}_{p \in S}$ is a finite set.
\end{proposition}

\begin{proof}
$(\Rightarrow)$ Assume that $f = \prod\limits_{p \in S}p^{a_p}$ exists in $D$.
Then, as $f\in D_n$ for some $n\in\mathbb{N}_0$, $f$ must be divided by a cyclotomic polynomial of odd degree in $D_n$ by Theorem \ref{q1}.
In other words, there exists $g\in D$ that is either a unit of $D$ or a finite product of prime elements of $D$
such that $f=g\prod\limits_{j=1}^l\Phi_{d_j}(X^{\frac{1}{2^n}})^{b_j}$ for some $g\in D$, $b_j\in\mathbb{N}$
and distinct odd integers $d_j\in\mathbb{N}$.
Since $\Phi_{d_j}(X^{\frac{1}{2^n}})^{b_j}=\prod\limits_{i\in\mathbb{N}}\Phi_{2d_j}(X^{\frac{1}{2^{n+i}}})^{b_j}$
by Corollary \ref{inf}, we have (1) and (2).

$(\Leftarrow)$ Conversely, suppose that (1) and (2) hold. By Proposition \ref{lemma5.5},
without loss of generality, we may assume that $l=1$. Furthermore,
for simplicity, let $\mathcal{B}_1 = \mathcal{B}$, $b_1= b$ and $d_1=d$.
Hence, if we let $\mathcal{C}' = \{p \in \mathcal{B} \mid a_p \neq b\}$, then
$|\mathcal{C}'| < \infty$, so $\prod\limits_{p \in \mathcal{C}'}p^{a_p - b} \in D.$
Moreover, $\prod\limits_{p \in T \setminus \mathcal{C} \cup \mathcal{C}'}p^{b} = \Phi_{d}(X^{\frac{1}{2^{n}}})^{b} \in D$
by assumption and Corollary \ref{inf} and $\prod\limits_{p \in \mathcal{C}}p^{a_p} \in D$ because $|\mathcal{C}| < \infty$.
Thus, by Proposition \ref{lemma5.5},
$$\prod\limits_{p \in S}p^{a_p}
= \bigg(\Phi_{d}(X^{\frac{1}{2^{n}}})^{b}\bigg)\bigg(\prod\limits_{p \in \mathcal{C}'}p^{a_p - b}\bigg)\bigg(\prod\limits_{p \in \mathcal{C}}p^{a_p}\bigg)$$
is a nonzero nonunit of $D$.
\end{proof}

Let $p_1, \dots, p_n$ be prime elements of an integral domain $E$
and $e_1, \dots , e_n$ positive integers. It is easy to see that $(\prod\limits_{i=1}^np_i^{e_i})E = \bigcap\limits_{i=1}^np_i^{e_i}E$
and $\prod\limits_{i=1}^np_i^{e_i}$ is a greatest common divisor of $\{p_1^{e_1}, \dots, p_n^{e_n}\}$.
The next result shows that this is true of a product of infinitely many prime elements in
$D= \bigcup\limits_{n\in\mathbb{N}_0}\mathbb{Q}[X^{\frac{1}{2^n}}, X^{-\frac{1}{2^n}}]$.

\begin{corollary}
\label{finaltheorem}
Let the notation be as in Notation \ref{notation} with $F=\mathbb{Q}$, $S$
the set of monic polynomials of $R$ that are prime elements of $D$,
and $\{k_p\}_{p \in S}$ a set of nonnegative integers. Then the following statements are equivalent
for a nonzero nonunit $f$ of $D$.
\begin{enumerate}[font=\normalfont]
    \item $f = u_f\prod\limits_{p \in S}p^{k_p}$.
    \item $fD = \bigcap\limits_{p \in S}p^{k_p}D$.
    \item $f$ is a least common multiple of $\{p^{k_p}\}_{p\in S}$ in $D$.
\end{enumerate}
\end{corollary}

\begin{proof}
(1) $\Rightarrow$ (2) Let $f = u_f\prod\limits_{p \in S}p^{k_p}$. Then $p^{k_p}|f$ for all $p \in S$
by Definition \ref{defi5.2}, and hence $f \in \bigcap\limits_{p \in S}p^{k_p}D$. Thus,
$fD \subseteq \bigcap\limits_{p \in S}p^{k_p}D$. For the reverse containment, let
$g \in \bigcap\limits_{p \in S}p^{k_p}D$ be a nonzero nonunit of $D$.
Then, by Theorem \ref{theorem5.11}, $g = u_g\prod\limits_{p \in S}p^{l_p}$, where $l_p \in \mathbb{N}_0$ for all $p \in S$.
Moreover, $g \in p^{k_p}D$ implies $l_p \geq k_p$ for all $p \in S$.
Now let $a_p = l_p-k_p$ for all $p \in S$. Then $\{p \in S \mid a_p > 0\}$ is either a
finite set or an infinite set satisfying the conditions (1) and (2) of Proposition \ref{lastcoro},
so if we let $h = \prod\limits_{p \in S}p^{a_p}D$, then $h \in D$ by Proposition \ref{lastcoro}.
Hence, by Proposition \ref{lemma5.5}, $g = fh \in fD$. Therefore,
$\bigcap\limits_{p \in S}p^{k_p}D \subseteq fD$.

(2) $\Leftrightarrow$ (3) This follows from \cite[Theorem 2.1]{a00}.

(3) $\Rightarrow$ (1)
Assume that (3) holds. Then clearly $p^{k_p}\mid f$ in $D$ for each $p\in S$
and $f$ satisfies the condition (2) of Definition \ref{defi5.2}.
Suppose that $p^{k_p+1}\mid f$ in $D$ for some $p\in S$ and let $h=p^{-1}f$.
Then $h\in D$ and  $q^{k_q}\mid h$ in $D$ for each $q\in S$ because $q^{k_q}D$ is a $qD$-primary ideal of $D$
by Proposition \ref{prop3.1}(1) and \cite[Proposition 4.2]{am69}.
Hence, $h$ is a common multiple of $\{p^{k_p}\}_{p\in S}$, so we have $h=fl$ for some $l\in D$. It follows that $p$ is a unit of $D$, a contradiction.
Thus, $p^{k_p+1}\nmid f$ in $D$ for each $p\in S$. In other words, $f$ satisfies (1) of Definition \ref{defi5.2}.
\end{proof}

We end this section with a product of infinitely many prime elements analog of the following well-known result:
Let $p_1, \dots, p_n$ be nonassociated prime elements of an integral domain $E$,
$e_1, \dots , e_n$, $k_1, \dots , k_n$ nonnegative integers, $f = \prod\limits_{i=1}^np_i^{e_i}$
and $g = \prod\limits_{i=1}^np_i^{k_i}$.
Then $f, g \in E$ and $\prod\limits_{i=1}^np_i^{\min\{e_i, k_i\}}$ (resp.,
 $\prod\limits_{i=1}^np_i^{\max\{e_i, k_i\}}$) is a greatest common divisor (resp., a least common multiple) of $f$ and $g$ in $E$.

\begin{corollary}
Let the notation be as in Notation \ref{notation} with  $F=\mathbb{Q}$ and
$S$ the set of monic polynomials of $R$ that are prime elements of $D$.
If $f=u_f\prod\limits_{p \in S}p^{a_p}$ and $g=u_g\prod\limits_{p \in S}p^{b_p}$ are nonzero nonunits of $D$,
then $h= \prod\limits_{p \in S}p^{\min\{a_p,b_p\}}$ $($resp., $l= \prod\limits_{p \in S}p^{\max\{a_p,b_p\}})$
is a greatest common divisor  $($resp., a least common multiple$)$ of $f$ and $g$ in $D$.
\end{corollary}

\begin{proof}
By Corollary \ref{finaltheorem}, we have
\begin{align*}
fD\cap gD&=(\bigcap\limits_{p \in S}p^{a_p}D)\cap (\bigcap\limits_{p \in S}p^{b_p}D)\\
&=\bigcap\limits_{p \in S}(p^{a_p}D\cap p_p^{b_p}D)\\
&=\bigcap\limits_{p \in S}p^{\max\{a_p,b_p\}}D.
\end{align*}
Since $D$ is a B{\'e}zout domain by Proposition \ref{prop3.1},
there exists $l'\in D$ such that $l'D=fD\cap gD$  \cite[Theorem 2.1]{a00}.
It follows that $l = \prod\limits_{p \in S}p^{\max\{a_p,b_p\}}$ is a least common multiple of $f$ and $g$ by Corollary \ref{finaltheorem}.
Now, let $h'\in D$ be a greatest common divisor of $f$ and $g$ in $D$. Then we have $fgD=lh'D$ \cite[Theorem 2.1(2)]{a00},
so $h'=u_{h'}\prod\limits_{p \in S}p^{\min\{a_p,b_p\}}$
by Proposition \ref{lemma5.5} and Theorem \ref{theorem5.11}.
Hence, $h=\prod\limits_{p \in S}p^{\min\{a_p,b_p\}}$ is a greatest common divisor of $f$ and $g$.
\end{proof}

Motivated by the results of this section, in \cite{cc26}, we
introduced the notion of formal unique factorization domains and studied
the ring-theoretic characterizations of formal UFDs systematically.
Then $D= \bigcup\limits_{n\in\mathbb{N}_0}\mathbb{Q}[X^{\frac{1}{2^n}}, X^{-\frac{1}{2^n}}]$
is a formal UFD by Theorem \ref{theorem5.11} and Corollary \ref{finaltheorem}
(see the remark after Corollary 2.3 in \cite{cc26}).

\vspace{.2cm}
\noindent
\textbf{Acknowledgements.}
The first author was supported by Basic Science Research Program through the National Research Foundation of Korea (NRF)
funded by the Ministry of Education (2017R1D1A1B06029867). The second author was supported by the National Research Foundation of Korea (NRF) grant funded by the Korea government (MSIT)(No. 2022R1C1C2009021).

\end{document}